\documentclass [a4paper,10pt]{amsart}
\usepackage{longtable}
\usepackage{hyperref}
\usepackage[T1]{fontenc}
\usepackage[utf8]{inputenc}
\usepackage[english]{babel}
\usepackage{textcomp}
\usepackage{dsfont}
\usepackage{latexsym}
\usepackage{amssymb}
\usepackage{amsthm}
\usepackage{amsmath}
\DeclareMathAlphabet{\mathpzc}{OT1}{pzc}{m}{en}
\usepackage{yfonts}
\usepackage{xfrac}
\usepackage{newlfont}
\usepackage{graphicx}
\usepackage{mathtools}
\usepackage{comment}
\usepackage{indentfirst}
\usepackage{braket}
\usepackage{mathrsfs}
\usepackage{xcolor}
\usepackage{fixmath}

\usepackage{etoolbox}

\usepackage{scalerel}[2014/03/10]
\usepackage[usestackEOL]{stackengine}
\newcommand{\dashint}{\,\ThisStyle{\ensurestackMath{%
			\stackinset{c}{.2\LMpt}{c}{.5\LMpt}{\SavedStyle-}{\SavedStyle\phantom{\int}}}%
		\setbox0=\hbox{$\SavedStyle\int\,$}\kern-\wd0}\int}

\DeclareMathOperator{\sgn}{sgn}

\DeclareMathOperator{\pr}{pr}

\DeclareMathOperator{\supp}{Supp}

\DeclareMathOperator{\tr}{Tr}

\DeclareMathOperator{\Hol}{Hol}

\DeclareMathOperator{\Int}{Int}

\newcommand{\Pol}{\mathrm{Pol}}
\newcommand{\CR}{\mathrm{CR}}

\renewcommand{\Re}{\mathrm{Re}\,}
\renewcommand{\Im}{\mathrm{Im}\,}
\newcommand{\Supp}[1]{\supp\left( #1\right) }
\newcommand{\Pfaff}{\mathrm{Pf}}
\newcommand{\ee}{\mathrm{e}}

\newcommand{\dd}{\mathrm{d}}

\DeclarePairedDelimiter{\abs}{\lvert}{\rvert}

\DeclarePairedDelimiter{\norm}{\lVert}{\rVert}

\let\originalleft\left
\let\originalright\right
\renewcommand{\left}{\mathopen{}\mathclose\bgroup\originalleft}
\renewcommand{\right}{\aftergroup\egroup\originalright}

\newcommand{\N}{\mathds{N}}

\newcommand{\C}{\mathds{C}}

\newcommand{\R}{\mathds{R}}

\newcommand{\Es}{\mathscr{E}}

\newcommand{\Hs}{\mathscr{H}}

\newcommand{\Dc}{\mathcal{D}}
\newcommand{\Ec}{\mathcal{E}}
\newcommand{\Fc}{\mathcal{F}}

\newcommand{\Hc}{\mathcal{H}}

\newcommand{\Kc}{\mathcal{K}}
\newcommand{\Lc}{\mathcal{L}}
\newcommand{\cM}{\mathcal{M}}
\newcommand{\Nc}{\mathcal{N}}
\newcommand{\Oc}{\mathcal{O}}
\newcommand{\Pc}{\mathcal{P}}

\newcommand{\Rc}{\mathcal{R}}
\newcommand{\Sc}{\mathcal{S}}

\newcommand{\Xc}{\mathcal{X}}

\newcommand{\Zc}{\mathcal{Z}}

\newcommand{\meg}{\leqslant}
\newcommand{\Meg}{\geqslant}
\newcommand{\eps}{\varepsilon}
\renewcommand{\phi}{\varphi}

\newcommand{\Lin}{\mathscr{L}}

\title{Paley--Wiener--Schwartz Theorems on Quadratic CR Manifolds}

\date{}
\begin{document}

\theoremstyle{definition}
\newtheorem{deff}{Definition}[section]

\newtheorem{oss}[deff]{Remark}

\newtheorem{ass}[deff]{Assumptions}

\newtheorem{nott}[deff]{Notation}

\theoremstyle{plain}
\newtheorem{teo}[deff]{Theorem}

\newtheorem{lem}[deff]{Lemma}

\newtheorem{prop}[deff]{Proposition}

\newtheorem{cor}[deff]{Corollary}

\author[M. Calzi]{Mattia Calzi}

\address{Dipartimento di Matematica, Universit\`a degli Studi di
	Milano, Via C. Saldini 50, 20133 Milano, Italy}
\email{{\tt mattia.calzi@unimi.it}}

\thanks{{\em Math Subject Classification 2020:} Primary: 32V20; Secondary: 32A10.}
\thanks{The author is a member of the	Gruppo Nazionale per l'Analisi Matematica, la Probabilit\`a e le	loro	Applicazioni (GNAMPA) of the Istituto Nazionale di Alta Matematica	(INdAM) and is  partially supported by the 2020	GNAMPA grant {\em Fractional Laplacians and subLaplacians on Lie groups and trees}.}

\begin{abstract} 
	Given a quadratic CR manifold $\cM$ embedded in a complex space, we study Paley--Wiener--Schwartz theorems for spaces of Schwartz functions and tempered distributions on $\cM$.
\end{abstract}
\maketitle

\section{Introduction}

Let $K$ be a compact convex subset of $\R^n$, and denote by 
\[
H_{K}\colon \R^n\ni h\mapsto \sup_{\lambda\in -K} \langle \lambda, h\rangle\in [-\infty,\infty)
\] 
the supporting function of $K$.
Then, it is well known (cf., e.g.,~\cite[Theorem 7.3.1]{Hormander}) that for every $u\in \Sc'(\R^n)$ with $\Supp{\Fc u}\subseteq K$, where $\Fc$ denotes the Fourier transform on $\R^n$, there is a unique entire function $f$ on $\C^n$ such that $u$ is the restriction of $f$ to $\R$. In addition, 
\[
f(z)=\frac{1}{(2\pi)^n} \big\langle \Fc u, \ee^{i\langle\,\cdot\,,z\rangle}\big\rangle
\]
for every $z\in \C^n$ and there is a constant $C>0$ such that
\[
\abs{f(z)}\meg C (1+\abs{z})^N \ee^{H_{K}(\Im z)}
\]
for every $z\in \C^n$, where $N$ is the order of $\Fc u$ as a distribution. 
Conversely, given an entire function $f$ satisfying the above estimates, the Fourier transform of the restriction of $f$ to $\R^n$ is a distribution of order at most $N$ and supported in $K$. 
Furthermore, $u\in \Sc(\R^n)$ if and only if for every $N'\in \N$ there is $C_{N'}>0$ such that
\[
\abs{f(z)}\meg \frac{C_{N'}}{(1+\abs{z})^{N'}} \ee^{H_{K}(\Im z)}
\]
for every $z\in \C^n$.

In this paper we shall investigate the analogues of the above Paley--Wiener--Schwartz theorems when the role of $\R^n$ is replaced by a possibly curved submanifold, namely a quadratic (or quadric) CR submanifold, which we shall now describe (cf.~\cite{Boggess,PelosoRicci,PelosoRicci2}).

Let $E$ be a finite-dimensional complex vector space, $F$ a finite-dimensional real vector space, and $\Phi\colon E\times E\to F_\C$ a hermitian mapping, where $F_\C$ denotes the complexification of $F$. Then,
\[
\cM\coloneqq \Set{(\zeta,z)\colon E\times F_\C\colon \Im z- \Phi(\zeta)=0}
\]
is the quadratic CR submanifold associated with $\Phi$.
If the open conve cone
\[
\Lambda_+\coloneqq \Set{\lambda\in F'\colon \forall \zeta\in E\setminus \Set{0} \:\: \langle \lambda, \Phi(\zeta)\rangle>0},
\]
where $\Phi(\zeta)\coloneqq \Phi(\zeta,\zeta)$ for every $\zeta\in E$, is not empty, then $\cM$ may be interpreted as the \v Silov boundary of a Siegel domain of type II. More precisely, given any non-empty open convex cone $\Omega$ not containing affine lines  and such that its dual $\Omega'\coloneqq \Set{\lambda\in F'\colon \forall h\in \overline\Omega\setminus\Set{0}\:\: \langle \lambda,h\rangle>0}$ is contained in $\Lambda_+$, the open convex set
\[
D_\Omega\coloneqq\Set{(\zeta,z)\in E\times F_\C\colon \Im z- \Phi(\zeta)\in \Omega}
\]
is a Siegel domain of type II with \v Silov boundary $\cM$.

Quadratic CR submanifolds may therefore be considered as generalizations of the \v Silov boundaries of Siegel domains of type II, and several of their features have been investigated (cf., e.g.,~\cite{VergneRossi,Treves,PelosoRicci,PelosoRicci2}). 
Notice that  the sets $\cM+i h$, $h\in F$ foliate $E\times F_\C$, so that it is natural to replace $\Im z$, in the classical Paley--Wiener--Schwartz theorems, with 
\[
\rho(\zeta,z)\coloneqq \Im z- \Phi(\zeta).
\]
In addition, $\cM$ has a natural $2$-step nilpotent Lie group structure, so that the Fourier transform on $\cM$ may play a role in this more general situation, up to some extent. Since, unfortunately, the Fourier transform of general tempered distributions is not easily described when $\cM$ is not commutative, we shall rely on the Eucliden Fourier transform $\Fc_F$ on $F$ (which is the centre of $\cM$ when $\Phi$ is non-degenerate) when necessary. 
We shall then characterize the restrictions to $\cM$ of the entire functions $f$ such that there are $N,C>0$ such that
\[
\abs{f(\zeta,z)}\meg C(1+\abs{\zeta}+\abs{z})^N \ee^{H_{K}(\rho(\zeta,z))}
\]
for every $(\zeta,z)\in E\times F_\C$ (cf.~Theorem~\ref{teo:2}). In particular, we shall prove that,  if this estimate hold, then a similar estimate holds (possibily with different $C$ and $N$) with $K$ replaced by $K\cap \Pc$, where $\Pc$ is the polar of $\Phi(E)$. We shall also provide a structure theorem for $\Oc_K(\Nc)$, showing that one may always essentially reduce to the case in which $K\cap \Lambda_+$ has a non-empty interior (cf.~Proposition~\ref{prop:10}).
We may then characterize the $f$ as above whose restriction to $\cM$ is a Schwartz function as the functions $f$ such that for every $N\in \N$ there is $C_N>0$
\[
\abs{f(\zeta,z)}\meg \frac{C_N}{(1+\abs{z})^N} \ee^{H_{K}(\rho(\zeta,z))}
\]
for every $(\zeta,z)\in E\times F_\C$ (cf.~Theorem~\ref{teo:1}). As before, we shall prove that, if these estimates hold, then similar estimates hold with $1+\abs{z}$ replaced by $1+\abs{\zeta}+\abs{z}$, and $K$ replaced by $K\cap \overline{\Lambda_+}$.   In particular, if $\Phi$ is degenerate or $\Phi(E)$ is not contained in any proper closed convex cone (so that $\Lambda_+=\emptyset$), then no such functions arise (except for the zero function).
Even though these results are proved relying as far as possible to the abelian (sometimes even one-dimensional) case by means of suitable slicing procedures, we shall provide interpretations in terms of the intrinsic Fourier transform on $\cM$ whenever possible. 

We then proceed to study further some spaces of Schwartz functions on $\cM$ which are related to the preceding results. 

The paper is organized as follows. In Section~\ref{sec:2}, we recall some basic facts on quadratic CR manifolds, Fourier analysis thereon, and some basic notions of convex analysis.
In Section~\ref{sec:3} we prove a Paley--Wiener--Schwartz theorem for tempered CR distributions. In Section~\ref{sec:4} we prove a Paley--Wiener--Schwartz theorem for CR Schwartz functions.
In Section~\ref{sec:5} we study some auxiliary spaces of CR Schwartz functions. As an application, we derive some structure results for the spaces of tempered CR distributions treated in Section~\ref{sec:3}.

\section{Analysis on Quadratic CR Manifolds}\label{sec:2}

In this section we introduce our basic notation for quadratic CR manifolds, recall some basic facts on CR distributions thereon, develop some basic tools of Fourier analysis, define an auxiliary Rockland operator, and then define our notation of the polar and the supporting function of a convex set.

\subsection{Basic Definitions}

We fix a complex hilbertian space $E$ of dimension $n$, a real hilbertian space $F$ of dimension $m$, and a hermitian map $\Phi\colon E\times E\to F_\C$. 
The associated quadratic CR submanifold is then
\[
\cM\coloneqq \Set{(\zeta,x+i\Phi(\zeta))\colon \zeta\in E, x\in F }=\Set{(\zeta,z)\in E\times F_\C\colon \Im z-\Phi(\zeta)=0},
\]
where $F_\C$ denotes the complexification of $F$, while $\Phi(\zeta)\coloneqq\Phi(\zeta,\zeta)$ for every $\zeta\in E$. We define 
\[
\rho\colon E\times F_\C\ni (\zeta,z)\mapsto \Im z-\Phi(\zeta)\in F.
\]

We endow $E\times F_\C$ with the product
\[
(\zeta,z)\cdot (\zeta',z')\coloneqq (\zeta+\zeta', z+z'+2i \Phi(\zeta',\zeta))
\]
for every $(\zeta,z),(\zeta',z')\in E\times F_\C$, so that $E\times F_\C$ becomes a $2$-step nilpotent Lie group, and $\cM$ a closed subgroup of $E\times F_\C$. In particular, the identity of $E\times F_\C$ is $(0,0)$ and $(\zeta,z)^{-1}= (-\zeta,-z+2 i \Phi(\zeta))$ for every $(\zeta,z)\in E\times F_\C$.
It will be often convenient to identify $\cM$ with the $2$-step nilpotent Lie group $\Nc\coloneqq E\times F$, endowed with the  product
\[
(\zeta,x)(\zeta',x')\coloneqq (\zeta+\zeta', x+x'+2 \Im \Phi(\zeta,\zeta'))
\]
for every $(\zeta,x),(\zeta',x')\in \Nc$, by means of the isomorphism
\[
\pi\colon\Nc \ni (\zeta,x)\mapsto (\zeta, x+i \Phi(\zeta))\in E\times F_\C.
\]
In particular, the identity of $\Nc$ is $(0,0)$ and $(\zeta,x)^{-1}=(-\zeta,-x)$ for every $(\zeta,x)\in \Nc$.
Notice that, in this way, $\Nc$ acts holomorphically (on the left) on $E\times F_\C$.
In particular, if $X$ is a left-invariant differential operator on $\Nc$, then $\Xc\coloneqq \dd \pi(X)$ is a left-invariant differential operator on $E\times F_\C$, and
\[
(\Xc f)_h= X f_h
\]
for every $f\in C^1(E\times F_\C)$ and for every $h\in F$, where 
\[
f_h\colon \Nc\ni (\zeta,x)\mapsto f(\zeta,x+i \Phi(\zeta)+i h)\in \C.
\]
We shall say that $\Xc$ is the canonical extension of $X$ to $E\times F_\C$.

We denote by $\Rc$ the radical of $\Phi$, that is, $\Set{\zeta\in E\colon \Phi(\zeta,\,\cdot\,)=0}$.

\begin{prop}
	The centre of $\Nc$ is $\Rc\times F$, while the centre of $E\times F_\C$ is $\Rc\times F_\C$. The commutator subgroups of $\Nc$ and $E\times F_\C$ coincide with the vector space generated by $\Im \Phi(E\times E)$ (equivalently, by $\Re \Phi(E\times E)$, or by $\Phi(E)$).
\end{prop}

\begin{proof}
	It suffices to observe that
	\[
	[(\zeta,x),(\zeta',x')]_\Nc=(0,4\Im \Phi(\zeta,\zeta')) 
	\]
	for every $(\zeta,x),(\zeta',x')\in \Nc$, while
	\[
	[(\zeta,z),(\zeta',z')]_{E\times F_\C}=(0, 4 \Im\Phi(\zeta,\zeta'))
	\]
	for every $(\zeta,z),(\zeta',z')\in E\times F_\C$.
\end{proof}

\subsection{CR Distributions}

For every $v\in E$, we denote by $Z_v$ the left-invariant vector field on $\Nc$ which induces $\partial_{E,v}\coloneqq \frac 1 2(\partial_v-i\partial_{iv})$ at $(0,0)$.\footnote{We use this notation to stress the dependence of the Wirtinger derivatives  $\partial_{E,v}$ on the complex structure of the space $E$, since we shall consider different complex structures on (some subspaces of) $E$.}
Then,
\[
Z_v= \partial_{E,v}+i \Phi(v,\,\cdot\,) \partial_F
\]
for every $v\in E$. In other words, 
\[
(Z_v f)(\zeta,x)= (\partial_{E,v} f)(\zeta,x)+i(\partial_{\Re \Phi(\zeta,v)} f)(\zeta,x)+ (\partial_{\Im \Phi(\zeta,v)} f)(\zeta,x)
\]
for every $f\in C^1(\Nc)$ and for every $(\zeta,x)\in \Nc$.

Notice that, denoting by $\Zc_v$ the canonical extension of $Z_v$ to $E\times F_\C$, one has
\[
\overline{\Zc_v } f(\zeta,z)=\overline{\partial_{E,v} } f(\zeta,z)
\]
for every $f\in C^1(E\times F_\C)$ and for every $(\zeta,z)\in E\times F_\C$ \emph{such that $f(\zeta,\,\cdot\,)$ is holomorphic is a neighbourhood of $z$}.

Observe that the $Z_v$, $v\in E$, generate a left-invariant involutive subbundle of the tangent bundle of $\Nc$, which coincides with the CR structure induced by $E\times F_\C$ through the (real analytic) embedding $(\zeta,x)\mapsto (\zeta,x+i\Phi(\zeta))$ (cf.~\cite[Section 7.4]{Boggess}). We shall also say that $\Nc$, endowed with this CR structure, is a quadratic CR manifold (cf.~\cite{PelosoRicci,PelosoRicci2}).

A distribution $u$ on $\Nc$ is then said to be a CR distribution if $\overline{Z_v}u=0$ for every $v\in E$ (cf.~\cite[Sections 9.1 and 17.2]{Boggess}).
As observed in~\cite[Section 17.2]{Boggess}, if $u$ is a CR distribution on $\Nc$, then there is a unique $v\in C^\infty(E; \Dc'(F))$, where $\Dc'(F)$ denotes the space of distributions on $F$, such that 
\[
\langle u, \phi\otimes \psi \rangle = \int_E \phi(\zeta) \langle v(\zeta), \psi\rangle\,\dd \zeta
\]
for every $\phi \in C^\infty_c(E)$ and for every $\psi\in C^\infty_c(F)$. 

In order to put our results in a better perspective, we now report an extension theorem for CR functions on $\Nc$, which follows from~\cite[Theorem 1 of Section 14.2 and Theorem 1 of Section 15.3]{Boggess}.
This result can be extended to CR distributions, with distributional convergence to the boundary values. The convergence to the boundary values may be improved for several classes of distributions.

\begin{prop}
	Assume that the convex envelope of $\Phi(E)$ has a non-empty interior $\Omega$, and define $D\coloneqq \rho^{-1}(\Omega)$. Then, for every CR function $u\in C^1(\Nc)$ there is a unique $f\in \Hol(D)$ such that $f_{h}$ converges locally uniformly to $u$ for $h\to 0$, $h\in \Gamma$, for every  convex cone $\Gamma$ contained in $\Omega$ and not containing affine lines (and for $\Gamma=F$ if $\Omega=F$).\footnote{Note that~\cite[Theorem 1 of Section 15.3]{Boggess} actually states that $f_h\to u$ locally uniformly for $h\to 0$, $h\in \Omega$ (equivalently, that $f$ extends by continuity to $D\cup \rho^{-1}(0)$), but the proof only shows convergence for $h\to 0$, $h\in \Gamma$, where $\Gamma$ is a closed convex cone contained in $\Omega \cup\Set{0}$. A simple adaptation of the techniques presented in~\cite[Chatper 15]{Boggess} leads to the present formulation.}
\end{prop}

In particular, if $\Omega=F$, then every CR function (of class $C^1$) on $\Nc$ extends to an entire function on $E\times F_\C$. Nontheless, since we are going to consider tempered distributions on $\Nc$, this case  will be of little interest, since then the only CR tempered distributions on $\Nc$ are the CR polynomials, as a consequence of Theorem~\ref{teo:2} and Proposition~\ref{prop:6}.

\subsection{The Fourier Transform on $\Nc$}\label{sec:2:1}

We shall now describe (a representative of) all irreducible continuous unitary representations of $\Nc$ and present the associated Plancherel formula.  Cf.~also~\cite{PelosoRicci,CalziPeloso}.

Take $\lambda\in F'$, and denote by $\Rc_\lambda$ the radical of the hermitian form $\langle \lambda_\C, \Phi\rangle$. Then, there is a unique $J_\lambda\in GL( \Rc_\lambda^\perp)$ such that
\[
\langle \lambda_\C, \Phi(\zeta,i \zeta')\rangle= \langle \zeta\vert J_\lambda \zeta'\rangle
\]
for every $\zeta,\zeta'\in \Rc_\lambda^\perp$. We shall then define $J'_\lambda\coloneqq \abs{J_\lambda}^{-1} J_\lambda$, and denote by $E_\lambda$ the space $\Rc_\lambda^\perp$ endowed complex scalar product defined by
\[
(x+i y)\cdot_\lambda \zeta \coloneqq x \zeta+ i y J'_\lambda \zeta
\]
for every $x,y\in \R$ and for every $\zeta\in \Rc_\lambda^\perp$. We shall keep on $\Rc_\lambda^\perp$ the complex scalar product induced by $E$.

Then,\footnote{Here we depart slightly from the notation employed in~\cite{CalziPeloso}.} 
\[
\Phi_\lambda\colon (\zeta,\zeta')\mapsto \langle \lambda, \Im \Phi(J'_\lambda \zeta,\zeta')\rangle+i \langle \lambda, \Im \Phi( \zeta,\zeta')\rangle  
\]
is a positive hermitian form on $E_\lambda$. Observe that, since $J'_\lambda\in GL(\Rc_\lambda^\perp)$, $J'^2_\lambda=-I$, and $J'_\lambda$ is clearly skew-adjoint, $\Rc_\lambda^\perp$ is the direct sum of two orthogonal subspaces $E_{\lambda,+}$ and $E_{\lambda,-}$ such that $J'_\lambda=\pm i$ on $E_{\lambda, \pm}$.  In particular, $E_{\lambda,+}$ and $E_{\lambda,-}$ are also (complex) vector subspaces of $E_\lambda$.

Observe that $\Nc/\ker \lambda$ is isomorphic to the product of the $(2 (n-d_\lambda)+1)$-dimensional Heisenberg group  and  the abelian group $\Rc_\lambda$, where $d_\lambda\coloneqq \dim \Rc_\lambda$. We denote by $(\Rc_\lambda)'_\R$ the dual of the \emph{real} vector space subjacent to $\Rc_\lambda$. Therefore, the Stone--von Neumann theorem (cf., e.g.,~\cite[Theorem 1.50]{Folland}) implies that for every $\tau\in (\Rc_\lambda)'_\R$ there is (up to unitary equvalence) a unique irreducible continuous unitary representation $\pi_{\lambda,\tau}$ of $\Nc$ such that 
\[
\pi_{\lambda,\tau}(\zeta,x)= \ee^{-i \langle \lambda, x\rangle- i \langle \tau, \zeta \rangle} I
\]
for every $(\zeta,x)\in \Rc_\lambda \times F$. 

We may realize $\pi_{\lambda,\tau}$ as follows. Define $\Hs_\lambda \coloneqq \Hol (E_\lambda) \cap L^2(\nu_\lambda)$, where $\nu_\lambda \coloneqq \ee^{-2 \Phi_\lambda}\cdot \Hc^{2(n-d_\lambda)}$, where $\Hc^{2(n-d_\lambda)}$ denotes the (suitably normalized) $2(n-d_\lambda)$-dimensional Hausdorff measure on $E_\lambda$, and set
\[
\pi_{\lambda,\tau}(\zeta+\zeta',x)\psi(\omega)\coloneqq \ee^{-i \langle \lambda, x\rangle- i \langle \tau, \zeta' \rangle+ 2 \Phi_\lambda(\omega,\zeta)-\Phi_\lambda(\zeta)} \psi(\omega-\zeta)
\]
for every $\psi\in \Hs_\lambda$, for every $\omega\in E_\lambda$, for every $\zeta\in E_\lambda$, for every $\zeta'\in \Rc_\lambda$, and for every $x\in F$.
We shall also set
\[
\pi_{\lambda}\coloneqq \pi_{\lambda,0}
\]
for every $\lambda\in F'$.

Given a complex space $H$, we define $\partial_{H,v}\coloneqq \frac 1 2 (\partial_v-i \partial_{i v})$ for every $v\in H$.

\begin{prop}\label{prop:1}
	Take $\lambda\in F'$ and $\tau\in (\Rc_\lambda)'_\R$. Then, the following hold:
	\begin{itemize}
		\item[\textnormal{(1)}] $\pi_{\lambda,\tau}$ is an irreducible continuous unitary representation of $\Nc$ in $\Hs_\lambda$;
		
		\item[\textnormal{(2)}]  the set of polynomials on $E_\lambda$ is contained  in the space of smooth vectors $C^\infty(\pi_{\lambda,\tau})$ for $\pi_{\lambda,\tau}$, and is dense in $\Hs_\lambda$;\footnote{It is actually contained and dense in $C^\infty(\pi_{\lambda,\tau})$, but we shall not need this fact.}
		
		\item[\textnormal{(3)}]  $e_{\lambda,0}\coloneqq \sqrt{\frac{2^{n-d_\lambda} \abs{\Pfaff(\lambda)}}{\pi^{n-d_\lambda}}}\chi_{E_\lambda}$ is a unit vector in $\Hs_\lambda$, and 
		\[
		\langle \pi_{\lambda,\tau}(\zeta+\zeta',x)e_{\lambda,0}\vert e_{\lambda,0}\rangle = \ee^{-i\langle \lambda, x\rangle - i \langle \tau, \zeta'\rangle - \Phi_\lambda(\zeta)}
		\]
		for every $\zeta\in E_\lambda$, for every $\zeta'\in \Rc_\lambda$, and for every $x\in F$;
		
		\item[\textnormal{(4)}]  for all $v \in E_\lambda$, $\psi \in C^\infty(\pi_{\lambda,\tau})$, and $\omega\in E_\lambda$,
		\[
		\dd \pi_{\lambda,\tau}(\partial_{E_\lambda, v}) \psi(\omega)= - \partial_v \psi(\omega) \qquad \text{and} \qquad \dd \pi_{\lambda,\tau}(\overline{\partial_{E_\lambda, v}}) \psi(\omega)= 2 \Phi_\lambda(\omega, v)\psi(\omega).
		\]
	\end{itemize}
\end{prop}

\begin{proof}
	(1)--(3) This is proved as in~\cite[Proposition 1.15]{CalziPeloso}.
	
	(2) This follows from~\cite[Theorem 1.63]{Folland}.
	
	(4) This is a simple computation.
\end{proof}

Conversely, if $\pi$ is an irreducible continuous unitary representation of $\Nc$, then there is $\lambda\in F'$ such that $\pi(0,x)=\ee^{-i\langle \lambda, x \rangle} I$ for every $x\in F$, so that $\pi$ induces a representation of $\Nc/\ker \lambda$, and is therefore unitarily equivalent to a representation of the form $\pi_{\lambda,\tau}$ for some $\tau\in (\Rc_\lambda)'_\R$.

We define 
\[
d\coloneqq \min_{\lambda\in F'} d_\lambda=\min_{\lambda\in F'}\dim \Rc_\lambda
\]
and
\[
W\coloneqq \Set{\lambda\in F'\colon d_\lambda>d}.
\]

We denote by $\Lin(\Hs_\lambda)$ the space of (continuous) endomorphisms of $\Hs_\lambda$, and by $\Lin^2(\Hs_\lambda)$ the space of Hilbert--Schmidt endomorphisms of $\Hs_\lambda$, for every $\lambda\in F'$.

\begin{prop}\label{prop:2}
	The mapping
	\[
	L^1(\Nc) \ni f \mapsto (\pi_{\lambda,\tau}(f))\in \prod_{\lambda\in F'}\prod_{\tau\in (\Rc_\lambda)'_\R} \Lin(\Hs_\lambda)
	\]
	induces an isomorphism
	\[
	L^2(\Nc)\to \frac{2^{n-m-3d}}{\pi^{n+m+d}} \int_{F'}^\oplus \int_{(\Rc_\lambda)'_\R}^\oplus \Lin^2(\Hs_\lambda)\,\dd \tau \abs{\Pfaff(\lambda)}\,\dd \lambda.
	\]
	In particular,
	\[
	\norm{f}_{L^2(\Nc)}^2=\frac{2^{n-m-3d}}{\pi^{n+m+d}} \int_{F'} \int_{(\Rc_\lambda)'_\R} \norm{\pi_{\lambda,\tau}(f)}^2_{\Lin(\Hs_\lambda)}\,\dd \tau \abs{\Pfaff(\lambda)}\,\dd \lambda
	\]
	for every $f\in L^2(\Nc)$.
\end{prop}

\begin{proof}
	The last equality follows from~\cite[Section 2]{AstengoCowlingDiBlasioSundari}. The first assertion then follows from the general theory (cf, e.g.,~\cite[Theorem 18.8.2]{Dixmier} or~\cite[Section 4.3]{CorwinGreenleaf}).
\end{proof}

We define 
\[
\Lambda_+\coloneqq \Set{\lambda\in F'\colon \forall \zeta\in E\setminus \Set{0}\:\: \langle \lambda , \Phi(\zeta) \rangle>0}
\] 
and
\[
\Pc=\Phi(E)^\circ= \Set{\lambda\in F'\colon  \langle \lambda , \Phi \rangle\Meg 0}= \Set{\lambda\in F'\colon E_\lambda=E_{\lambda,+}}.
\]

\begin{prop}\label{prop:7}
	The set $\Pc$ is a closed convex cone, with interior
	\[
	\Set{\lambda\in F'\colon \forall \zeta\in E\setminus \Rc \:\: \langle \lambda, \Phi(\zeta)\rangle >0},
	\]
	which coincides with $\Lambda_+$ if and only if $\Phi$ is non-denegerate. If, otherwise, $\Phi$ is degenerate, then $\Lambda_+=\emptyset$.	
	
	If $\Pc$ has a non-emtpy interior, then  $\Phi$ is proper on $\Rc^\perp$ and the convex envelope of $\Phi(E)$ is closed. 
\end{prop}

\begin{proof}
	It is clear that $\Pc$ is a closed convex cone.
	Observe that, by~\cite[Corollary 17.1.2]{Rockafellar} the convex envelope $C$ of the cone $\Phi(E)$ equals 
	\[
	\Set{\sum_{j=0}^m \Phi(\zeta_j)\colon \zeta_1,\dots, \zeta_m\in E}.
	\]
	Observe that both $\Phi(E)$ and $C$ are left unchanged if $E$ is replaced by $\Rc^\perp$, while $\Lambda_+=\emptyset$ if $\Phi$ is degenerate. Hence, we may assume that $\Phi$ is non-degenerate.

	Assume first that the interior of $\Pc$ is non-empty. Then, $\overline C=\Pc^\circ$ does not contain lines, so that $+\colon \overline C\times \overline C\to \overline C$ is proper by~\cite[Corollary 1 to Proposition 11 of Chapter II, \S\ 6, No.\ 8]{BourbakiTVS}, so that $C=\Phi(E)+\cdots+\Phi(E)$ ($m$ times) is closed if $\Phi(E)$ is closed. In addition,~\cite[Corollary to Proposition 17 of Chapter II, \S\ 2, No.\ 7]{BourbakiTVS} implies that $\langle\lambda,h\rangle>0$ for every $\lambda$ in the interior of $\Pc$ and for every non-zero $h\in \overline C$, so that the interior of $\Pc$ is contained in
	\[
	\Pc'\coloneqq \Set{\lambda\in F' \colon \langle \lambda, h\rangle>0\quad \forall \lambda\in C \setminus \Set{0}}=\Set{\lambda\in F'\colon \forall h\in \Phi(E)\setminus \Set{0} \:\: \langle \lambda, h\rangle >0}.
	\] 
	In addition, the mapping $\lambda\circ \Phi$ is proper for every $\lambda$ in the interior of $\Pc$, so that $\Phi$ is proper by~\cite[Proposition 5 of Chapter I, \S\ 10, No.\ 1]{BourbakiGT1}. Hence, $\Phi(E)$ is closed, so that $C$ is closed as well by the previous remarks. Therefore, $\Pc'$ is open in $F'$ and contained in $\Pc$, so that it coincides with the interior of $\Pc$.
	In addition, the preceding remarks imply that $\lambda\circ \Phi$ is a non-degenerate positive hermitian form on $E$ for every $\lambda\in \Pc'$, so that $\Phi(\zeta)=0$ if and only if $\zeta=0$, and $\Lambda_+=\Pc'$.
	
	Conversely, assume that $\Lambda_+\neq \emptyset$, and observe that if $\lambda\in \Lambda_+$ and $\lambda'\in F'$ is such that 
	\[
	\max_{\abs{\zeta}=1} \abs{\langle \lambda', \Phi(\zeta)\rangle}<\min_{\abs{\zeta}=1} \langle \lambda, \Phi(\zeta)\rangle,
	\]
	then $\lambda+\lambda'\in \Lambda_+$. Hence, $\Lambda_+$ is an open subset of $F'$ contained in $\Pc$, so that the preceding remarks imply that it coincides with the interior of $\Pc$.
\end{proof}

We denote by 
\[
P_{\lambda,0}
\]
the orthoprojector of $\Hs_\lambda$ onto the space of constant functions (that, is the space generated by $e_{\lambda,0}$), for every $\lambda\in F'$.

\begin{prop}\label{prop:3}
	Take a CR $f\in L^2(\Nc)$. Then,
	\[
	\pi_{\lambda,\tau}(f)= \chi_{\Lambda_+}(\lambda) \pi_{\lambda,\tau}(f) P_{\lambda,0}
	\]
	for almost every $\lambda\in F'$ and for almost every $\tau\in (\Rc_\lambda)'_\R$. In particular, $f=0$ if $\Lambda_+=\emptyset$.
	
	Conversely, if $g\in L^2(\Nc)$ and 
	\[
	\pi_{\lambda,\tau}(g)= \chi_{\Lambda_+}(\lambda) \pi_{\lambda,\tau}(g) P_{\lambda,0}
	\]
	for  almost every $\lambda\in F'$ and for almost every $\tau\in (\Rc_\lambda)'_\R$, then $g$ is CR.
\end{prop}

\begin{proof}
	\textsc{Step I.} Let us prove that 
	\[
	\pi_{\lambda,\tau}(f)= \chi_{\Pc}(\lambda) \pi_{\lambda,\tau}(f) P_{\lambda,0}
	\]
	for almost every $\lambda\in F'$ and for almost every $\tau\in (\Rc_\lambda)'_\R$.
	
	Fix a representative $\eta$ of $\lambda \mapsto [\tau \mapsto \pi_{\lambda,\tau}(f)]$, and observe that, by assumption, there is a negligible subset $N$ of $F'$ such that
	\[
	\eta(\lambda)(\tau)\dd \pi_{\lambda, \tau}(\overline{ Z_v})=0
	\]
	for  every $\lambda\in F'\setminus N$, for almost every $\tau\in (\Rc_\lambda)'_\R$, and for every $v$ in a countable dense subset of $E$, hence for every $v\in E$ by continuity. Take $\lambda\in \Pc \setminus N$. Then, 
	\[
	0=\eta(\lambda)(\tau)\dd \pi_{\lambda, \tau}(\overline{ Z_v})=2\eta(\lambda)(\tau)\langle \lambda, \Phi(\,\cdot\,,v)\rangle
	\]
	for every $v\in E$ and for almost every $\tau\in (\Rc_\lambda)'_\R$. Notice that, since $\langle \lambda,\Phi\rangle$ is (positive and) non-degenerate on $E_\lambda$, the $\langle \lambda, \Phi(\,\cdot\,,v)\rangle$, as $v$ runs through $\Rc_\lambda^\perp$, exhaust the space of complex linear forms on $E_\lambda$.
	Since the (holomorphic) polynomials on $E_\lambda$ are dense in $\Hs_\lambda$ by Proposition~\ref{prop:1}, this implies that $\eta(\lambda)(\tau)=\eta(\lambda)(\tau)P_{\lambda,0}$ for almost every $\tau\in (\Rc_\lambda)'_\R$.

	Then, take $\lambda\in F'\setminus (\Pc \cup N)$, and observe that there is a non-zero $v\in E_{\lambda,-}$, so that
	\[
	0=\eta(\lambda,\tau)\dd \pi_{\lambda, \tau}(\overline{ Z_v})=-\eta(\lambda,\tau)\partial_v
	\]
	for almost every $\tau\in (\Rc_\lambda)'_\R$. Since $\partial_v C^\infty(\pi_{\lambda,\tau})$ contains the set of polynomials on $E_\lambda$, which is dense in $\Hs_\lambda$ by Proposition~\ref{prop:1}, this implies that $\eta(\lambda)(\tau)=0$ for almost every $\tau\in (\Rc_\lambda)'_\R$.

	\textsc{Step II.} Since $\Pc$ is the closure of $ \Lambda_+$ when $\Lambda_+\neq \emptyset$ by Proposition~\ref{prop:7}, in order to complete the proof it will suffice to show that $f=0$ if $\Lambda_+=\emptyset$.
	Observe first that, if $\Pc$ has an empty interior, the assertion follows from \textsc{step I}. Then, assume that $\Pc$ has a non-empty interior $\Int\Pc$.
	Since, by Proposition~\ref{prop:7},
	\[
	\Int \Pc=\Set{\lambda\in F'\colon \forall \zeta\in E\setminus \Rc\:\: \langle \lambda, \Phi(\zeta)\rangle>0},
	\]
	this implies that $\Rc_\lambda=\Rc$ for every $\lambda\in \Int\Pc$, so that $\Rc\neq \Set{0}$ since $\Lambda_+=\emptyset$.  Therefore, the mapping $\zeta'\mapsto f(\zeta+\zeta')$ belongs to $L^2(\Rc)\cap \Hol(\Rc)$  for almost every $\zeta\in \Rc^\perp$. Since clearly $L^2(\Rc)\cap \Hol(\Rc)=\Set{0}$, the arbitrariness of $\zeta$ implies that  $f=0$.
	
	\textsc{Step III.} Now, take $g\in L^2(\Nc)$ such that
	\[
	\pi_{\lambda,\tau}(g)= \chi_{\Lambda_+}(\lambda) \pi_{\lambda,\tau}(g) P_{\lambda,0}
	\]
	for  almost every $\lambda\in F'$ and for almost every $\tau\in (\Rc_\lambda)'_\R$. Observe that we may assume that $\Lambda_+\neq \emptyset$, so that $\Phi$ is non-degenerate and $d_\lambda=0$ for almost every $\lambda\in F'$. Then, Proposition~\ref{prop:1} implies that, for every $v\in E$,
	\[
	\pi_{\lambda,0}(\overline{Z_v} g)= 2\pi_{\lambda,0}(g) P_{\lambda,0}\langle \lambda, \Phi(\,\cdot\,,v)\rangle =0
	\]
	for almost every $\lambda\in \Lambda_+$, while clearly $\pi_{\lambda,0}(\overline{Z_v} g)=0$ for almost every $\lambda\in F'\setminus \Lambda_+$. Then, $\overline{Z_v}g=0$ for every $v\in E$.
\end{proof}

\begin{prop}\label{prop:5}
	Let $K$ be a closed   subset of $\overline{\Lambda_+}$, and take $f\in L^2(\Nc)$. Then, the following conditions are equivalent:
	\begin{enumerate}
		\item[\textnormal{(1)}] $\pi_{\lambda,\tau}(f)=\chi_{K}(\lambda) \pi_{\lambda,0}(f) P_{\lambda,0}$ for almost every $\lambda\in F'$ and for almost every $\tau\in (\Rc_\lambda)'_\R$;
		
		\item[\textnormal{(2)}] $f$ is CR and $\Fc_F(f(\zeta,\,\cdot\,))$ is supported in $K$ for almost every $\zeta\in E$. 
	\end{enumerate}
\end{prop}

\begin{proof}
	Notice that we may assume that $\Lambda_+\neq\Set{0}$, so that $d_\lambda=0$ for almost every $\lambda\in F'$.
	
	(1) $\implies$ (2).  Define $f_j\in L^2(\Nc)$ so that $\pi_\lambda(f_j)=\chi_{K_j}\pi_\lambda(f)$ for every $j\in \N$, where $(K_j)$ is an increasing sequence of compact subsets of $K$ which covers $K$ (cf.~Proposition~\ref{prop:2}). Observe that $f_j$ is continuous for every $j\in \N$, and that
	\[
	f_j(\zeta,x)=\frac{2^{n-m}}{\pi^{n+m}} \int_{K_j} \tr(\pi_\lambda(f)\pi_\lambda(\zeta,x)^*) \abs{\Pfaff(\lambda)}\,\dd \lambda =\frac{2^{n-m}}{\pi^{n+m}} \int_{K_j} \tr(\pi_\lambda(f)\pi_\lambda(\zeta,0)^*) \ee^{i\langle \lambda, x \rangle}\abs{\Pfaff(\lambda)}\,\dd \lambda
	\]
	for every $(\zeta,x)\in \Nc$. In particular,
	\[
	\Fc_F(f_j(\zeta,\,\cdot\,))(\lambda)=\frac{2^{n}}{\pi^n} \chi_{K_j}(\lambda) \tr(\pi_\lambda(f)\pi_\lambda(\zeta,0)^*) \abs{\Pfaff(\lambda)}
	\]
	for almost every $\lambda\in F'$.
	In addition, $(f_j)$ converges to $f$ in $L^2(\Nc)$, so that $(f_j(\zeta,\,\cdot\,))$ converges to $f(\zeta,\,\cdot\,)$ in $L^2(F)$ for almost every $\zeta\in E$, up to a subsequence. Hence, $\Fc_F(f(\zeta,\,\cdot\,))$ is supported in $K$ for almost every $\zeta\in E$. Since Proposition~\ref{prop:3} implies that $f$ is CR, this proves (2).
	
	(2) $\implies$ (1). Observe that
	\[
	\pi_\lambda(f)=\int_E \Fc_F(f(\zeta,\cdot\,))(\lambda)\pi_\lambda(\zeta,0)\,\dd \zeta=0
	\]
	for almost every $\lambda\in F'\setminus K$. The assertion then follows from Proposition~\ref{prop:3}.
\end{proof}

\subsection{An Auxiliary Rockland Operator}\label{sec:1:3}

We define 
\[
\Lc\coloneqq \left(\sum_{j} (Z_{v_j}\overline{Z_{v_j}}+\overline{Z_{v_j}} Z_{v_j})\right)^2-\sum_{k} U_k^2,
\]
where  $(v_j)$ is an orthonormal basis of $E$ over $\C$, while $(U_k)$ is an orthonormal basis of invariant vector fields on the central subgroup $F$ on $\Nc$. 
It is not hard to prove that $\Lc$ does not depend on the choice of $(v_j)$ and $(U_k)$.
If we endow $\Nc$ with the dilations $t\cdot (\zeta,x)\coloneqq (t^{1/2}\zeta, t x)$, then $\Lc$ is homogeneous of degree $2$.

\begin{prop}\label{prop:8}
	 For every $\lambda\in F'$ and for every $\tau \in (\Rc_\lambda)'_\R$, the following hold:
	 \begin{enumerate}
	 	\item[\textnormal{(1)}]  $\dd \pi_{\lambda,\tau}(\Lc)$ has a purely discrete spectum;\footnote{In other words, $\Hs_\lambda$ has an orthonormal basis consting of eigenvectors of $\dd \pi_{\lambda,\tau}(\Lc)$, and each eigenvalue of $\dd \pi_{\lambda,\tau}(\Lc)$ has finite multiplicity. }
	 	
	 	\item[\textnormal{(2)}] $e_{\lambda,0}$ is an eigenvalue of $\dd \pi_{\lambda,\tau}(\Lc)$, and
	 	\[
	 	\langle\dd \pi_{\lambda,\tau}(\Lc) e_{\lambda,0}\vert e_{\lambda,0}\rangle=\left(\left(2\tr_\C\abs{J_\lambda}+\frac 1 2\abs{\tau}^2\right)^2+\abs{\lambda}^2\right)=\min \sigma(\dd \pi_{\lambda,\tau}(\Lc));
	 	\]
	 	
	 	\item[\textnormal{(3)}]  $P_{\lambda,0}$ is the eigenprojector of $\dd \pi_{\lambda,\tau}(\Lc)$ corresponding to the eigenvalue $ \min \sigma(\dd \pi_{\lambda,\tau}(\Lc))$.
	 \end{enumerate} 
 	In particular, $\Lc $ is a positive Rockland operator.
\end{prop}

Recall that a left-invariant differential operator on a homogeneous group is said to be Rockland if it is homogeneous and is its image under every non-trivial irreducible continuous unitary representation  is one-to-one on the set of smooth vectors. Rockland operators may also be equivalently characterized as hypoelliptic homogeneous left-invariant differential operators, thanks to the celebrated solution of the Rockland conjecture by Helffer and Nourrigat (cf~\cite{HelfferNourrigat}).

\begin{proof}
	Fix $\lambda\in F'$ and $\tau\in (\Rc_\lambda)'_\R$, take an orthonormal basis $(v_j)$ of $E_\lambda $ which induce bases of $E_{\lambda,-}$ and $E_{\lambda,+}$, and denote by $\Delta_\lambda$ and $\Delta_F$ the standard (negative) Laplacians on $\Rc_\lambda$ and $F$, respectively. Then,
	\[
	\Lc= \left(\sum_{j} (Z_{v_j}\overline{Z_{v_j}}+\overline{Z_{v_j}} Z_{v_j})+\frac 1 2 \Delta_\lambda  \right)^2-\Delta_F.
	\]
	Next, define
	\[
	e_{\lambda,\alpha}\coloneqq \prod_{j} \Phi_\lambda(\,\cdot\,, v_j)^{\alpha_j}
	\]
	for every $\alpha\in \N^{n-d_\lambda}$, and observe that $(w_\alpha)$ is a \emph{total} orthogonal family of elements of $\Hs_\lambda$.\footnote{Orthogonality is proved easily. Since clearly $(w_\alpha)$ generates the set of holomorphic polynomials on $E_\lambda$, the assertion then follows from Proposition~\ref{prop:1}.} In addition, Proposition~\ref{prop:1} implies that
	\[
	\dd \pi_{\lambda,\tau}(\Lc) e_{\lambda,\alpha}=  \left[\left(2\sum_{j} \Phi_\lambda(v_j) (1+2 \alpha_j)+\frac 1 2 \abs{\tau}^2 \right)^2+\abs{\lambda}^2\right] e_{\lambda,\alpha}
	\]
	for every $\alpha\in \N^{n-d_\lambda}$. Observe that $\Phi_\lambda(v_j)>0$ for every $j$, since $\Phi_\lambda$ is positive and non-degenerate.
	It then follows that $\dd \pi_{\lambda,\tau}(\Lc)$ is positive and has a purely discrete spectrum, and that its least eigenvalue is $ \left(2\sum_{j} \Phi_\lambda(v_j) +\abs{\tau}^2 \right)^2+\abs{\lambda}^2$, with corresponding eigenspace $\R e_{\lambda,0}$. It only remains to show that $\sum_{j} \Phi_\lambda(v_j)=\tr \abs{J_\lambda}$. This, in turn, is a consequence of the fact that
	\[
	\Phi_\lambda(v)= \langle v\vert \abs{J_\lambda} v \rangle
	\]
	for every $v\in E_\lambda$, which is readily established.
\end{proof}

Thanks to Proposition~\ref{prop:8}, one may develop a functional calculus for $\Lc$, considered as an essentially self-adjoint operator on $L^2(\Nc)$ with domain $C^\infty_c(\Nc)$. In particular, for every bounded measurable function $m\colon \R\to \C$ there is a unique $\Kc(m)\in \Sc'(\Nc)$ such that $m(\Lc) f=f*\Kc(m)$ for every $f\in \Sc(\Nc)$. Thanks to a celebrated result by Hulanicki (cf.~\cite{Hulanicki}), $\Kc(\Sc(\R))\subseteq \Sc(\Nc)$. Cf.~\cite{Martini2, Martini, Calzi, Calzi3} for more information on the mapping $\Kc$. 

In particular, if $m\in \Sc(\R)$, $\lambda\in F'$, and $\tau\in (\Rc_\lambda)'_\R$, then
\[
\pi_{\lambda,\tau}(\Kc(m))= m(\dd \pi_{\tau,\lambda}(\Lc)),
\]
thanks to~\cite[Proposition 3.7]{Martini}, since nilpotent groups are amenable.

\subsection{Operations on Convex Sets}

\begin{deff}
	We define the polar $A^\circ$ of a subset $A$ of $F$ by
	\[
	A^\circ\coloneqq \Set{\lambda\in F'\colon \forall v\in A \:\: \langle \lambda,v\rangle\Meg -1}.
	\]
	Analogously, we define the polar $B^\circ $ of a subset $B$ of $F'$,
	\[
	B^\circ \coloneqq \Set{v\in F\colon \forall \lambda\in B \:\: \langle \lambda,v \rangle\Meg -1}.
	\]
\end{deff}

Observe that $A^{\circ}$ is a closed convex subset of $F'$, and that $A^{\circ \circ}$ is the closed convex envelope of $A\cup \Set{0}$, thanks to the bipolar theorem~\cite[Theorem 1 of Chapter II, \S\ 6, No.\ 3]{BourbakiTVS}. If $A$ is a cone, then $A^\circ= \Set{\lambda\in F'\colon \forall v\in A \:\: \langle\lambda,v \rangle \Meg 0}$. In particular, $A^\circ$ is a closed convex cone. If $A$ is a vector subspace of $F$, then $A^\circ =\Set{\lambda \in F'\colon \forall v \in A \:\: \langle \lambda,v \rangle=0}$.

\begin{deff}
	Let $K$ be a closed convex subset of $F'$. We define the supporting function of $K$ as
	\[
	H_K\colon F\ni v \mapsto \sup_{\lambda\in K} (-\langle \lambda, v \rangle)\in [-\infty,\infty].
	\]
\end{deff}

Then, $H_K$ is a lower semi-continuous, positively homogeneous, sub-additive convex function, and it is finite and continuous when $K$ is non-emtpy and compact (cf.~\cite[Exercise 9 of Chapter II, \S\ 6]{BourbakiTVS} and~\cite[Section 4.3]{Hormander}, where a different convention is adopted). Further, $H_\emptyset(v)=\infty<H_K(v)$ for every non-empty closed convex subset $K$ of $F'$ and for every $v\in F$. 
In addition,
\[
K=\Set{\lambda \in F'\colon \forall v\in F \:\: \langle \lambda, v \rangle \Meg -H_K(v)}
\]
(cf.~\cite[Theorem 4.3.2]{Hormander}).
Observe that, if $K$ contains $\Set{0}$, then
\[
H_{K}(v)=\inf\Set{r \in\R\colon v\in r K^\circ},
\]
so that $H_{K}$ may be interpreted as the Minkowski functional (or gauge) associated with  $K^\circ$.

\begin{lem}\label{lem:1}
	Let $K$ be a closed convex subset of $F'$, and let $F_1$ be a closed subspace of $ F$. Denote by $K_1$ the orthogonal projection of $K$ into $F_1^{\circ\perp}$, identified with the dual of $F_1$. Then,
	\[
	H_K(v)=H_{K_1}(v)
	\] 
	for every $v\in F_1$.
\end{lem}

\begin{proof}
	It suffices to apply the definition, observing that $\langle\lambda+\lambda', v \rangle=\langle \lambda, v \rangle$ for every $\lambda\in F_1^{\circ \perp}$, for every $\lambda'\in F_1^\circ$, and for every $v\in F_1$.  
\end{proof}

\section{General Paley--Wiener--Schwartz Theorems}\label{sec:3}

In this section, we prove a Paley--Wiener--Schwartz theorem for CR tempered distributions. Some refinements will be considered in Section~\ref{sec:5}.

\begin{deff}
	Let $K$ be a compact subset of $F'$.
	We define $\widetilde \Oc_K(\Nc)$ as the space of CR $u\in \Sc'(\Nc)$ such that $(I\otimes\Fc_F) u$ is supported in $E\times K$.\footnote{Here, $I\otimes \Fc_F\colon \Sc'(\Nc)\to \Sc'(E\times F')$ is defined by means of the canonical identifications $\Sc'(\Nc)\cong \Sc'(E)\widehat \otimes \Sc'(F)$ and $\Sc'(E\times F')\cong \Sc'(E)\widehat \otimes \Sc'(F')$.}
	
	We define $\Oc_K(\Nc)$ as the space of $f\in C^\infty(\Nc)$ such that $f\cdot \Hc^{2n+m}\in \widetilde \Oc_K(\Nc)$ and such that there is $k\in \Nc$ such that for every left-invariant differential operator $X$ on $\Nc$ the function $(1+\abs{\,\cdot\,})^{-k} X f$ is bounded.
\end{deff}

\begin{prop}
	Let $K$ be a compact subset of $\Nc$. Then, $\Oc_K(\Nc)$ is the space of $f\in C^\infty(\Nc)$ such that the following hold:
	\begin{enumerate}
		\item[\textnormal{(1)}] there is $k\in \Nc$ such that for every left-invariant differential operator $X$ on $\Nc$ the function $(1+\abs{\,\cdot\,})^{-k} X f$ is bounded;
		
		\item[\textnormal{(2)}] $\overline{Z_v} f=0$ for every $v\in E$;
		
		\item[\textnormal{(3)}] $\Fc_F(f(\zeta,\,\cdot\,))$ is supported in $K$ for every $\zeta\in E$.
	\end{enumerate}
\end{prop}

\begin{proof}
	It suffices to show that, if $f$ satisfies (1) and (2), then $(I\otimes \Fc_F) f$ is supported in $E\times K$ if and only if $\Fc_F(f(\zeta,\,\cdot\,))$ is supported in $K$ for every $\zeta\in E$. Since clearly 
	\[
	\langle (I\otimes \Fc_F) f, \phi\otimes \psi\rangle= \langle f, \phi\otimes \Fc_F \psi\rangle= \int_E \phi(\zeta) \langle f(\zeta,\,\cdot\,), \Fc_F\psi\rangle \,\dd \zeta=\int_E \phi(\zeta) \langle \Fc_F(f(\zeta,\,\cdot\,)), \psi\rangle \,\dd \zeta
	\]
	for every $\phi\in \Sc(E)$ and for every $\psi\in \Sc(F)$, and since the mapping 
	\[
	\zeta\mapsto \langle \Fc_F(f(\zeta,\,\cdot\,)), \psi\rangle =\langle f(\zeta,\,\cdot\,), \Fc_F\psi\rangle 
	\]
	is continuous, the assertion follows.
\end{proof}

\begin{teo}\label{teo:2}
	Let $K$ be a compact convex subset of $F'$. Then, for every $f\in \Hol(E\times F_\C)$ such that there are $N,C>0$ such that
	\[
	\abs{f(\zeta,z)}\meg C(1+\abs{\zeta}^2+\abs{z})^N \ee^{H_{K}(\rho(\zeta,z))}
	\]
	for every $(\zeta,z)\in E\times F_\C$, one has $f_0\in \Oc_{K}(\Nc)$.
	
	Conversely,  for every $u\in \widetilde \Oc_K(\Nc)$ there is a unique  $f\in \Hol(E\times F_\C)$ such that $f_0\cdot \Hc^{2n +m}=u$. In addition, $f$ satisfies the preceding estimates and
	\[
	f(\zeta,z)=\frac{1}{(2\pi)^m} \int_K \Fc_F [f_0(\zeta,\,\cdot\,)](\lambda)\ee^{i\langle \lambda_\C, z\rangle}\,\dd \lambda
	\]
	for every $(\zeta,z)\in E\times F_\C$.
\end{teo}

In particular, $\widetilde \Oc_K(\Nc)=\Oc_K(\Nc)\cdot \Hc^{2n+m}$.

\begin{proof}
	Take $f\in \Hol(E\times F_\C)$ such that there are $N,C>0$ such that
	\[
	\abs{f(\zeta,z)}\meg C(1+\abs{\zeta}^2+\abs{z})^N \ee^{H_{K}(\rho(\zeta,z))}
	\]
	for every $(\zeta,z)\in E\times F_\C$. 
	Let $X$ be a left-invariant differential operator on $\Nc$, and denote by $\Xc$ its canonical extension to $E\times F_\C$. Observe that, by Cauchy estimates and the invariance of $\Xc$, there is a constant $C'>0$ such that
	\[
	\begin{split}
		\abs{(X f_0)(\zeta,x)}&=\abs{\Xc [f((\zeta,x+i \Phi(\zeta))\,\cdot\,)](0,0)}\\
		&\meg C' \sup_{\abs{(\zeta',z')}\meg 1} \abs{f((\zeta,x+i\Phi(\zeta))\cdot (\zeta',z'))}\\
		&\meg C C'\sup_{\abs{(\zeta',z')}\meg 1} (1+\abs{\zeta+\zeta'}^2+\abs{z'+x+i\Phi(\zeta)+2 i\Phi(\zeta',\zeta)})^N \ee^{H_{K}(\rho(\zeta',z'))}
	\end{split}
	\]
	for   every $(\zeta,x)\in \Nc$. Therefore, there is a constant $C''>0$ such that
	\[
	\abs{(X f_0)(\zeta,x)}\meg C''(1+\abs{\zeta}^2+\abs{x})^N 
	\]
	for every $(\zeta,x)\in \Nc$.

	Now, fix $\zeta\in E$, and observe that $f_0(\zeta,\,\cdot\,)$ is the restriction of $f^{(\zeta)}\coloneqq f(\zeta,\,\cdot\,+i\Phi(\zeta))$ to $F$. In addition,
	\[
	\begin{split}
		\abs{f^{(\zeta)}(z)}&\meg C\abs{(1+\abs{\zeta}^2+\abs{z+i\Phi(\zeta)})^N} \ee^{H_{K}(\Im z)}\\
		&\meg C (1+\abs{\zeta}^2+\abs{\Phi(\zeta)})^N(1+\abs{z})^N \ee^{H_{K}(\Im z)}
	\end{split}
	\]
	for every $z\in F_\C$. Therefore, the usual Paley--Wiener--Schwartz theorem (cf., e.g.,~\cite[Theorem 7.3.1]{Hormander}) shows that $\Fc_F(f_0(\zeta,\,\cdot\,))$ is (a distribution or order at most $N$) supported in $K$. Further, since $f$ is holomorphic, it is clear that $f_0$ is CR. Thus, $f_0\in \Oc_K(\Nc)$.
	
	Conversely, take $u\in \widetilde \Oc_K(\Nc)$. For every $\phi\in \Sc(E)$ and for every $v\in \Sc'(\Nc)$, define $v_\phi\in \Sc'(F)$ so that $\langle v_\phi, \psi\rangle\coloneqq \langle v, \phi\otimes \psi\rangle$, and observe that, by assumption, $\Fc_F u_\phi$ is supported in $K$ for every $\phi \in \Sc(E)$.
	Observe that the bilinear mapping\footnote{We denote by $\Dc$ the space $C^\infty_c$ endowed with its usual inductive limit topology.}
	\[
	\Dc(E)\times \Dc(F_\C) \ni (\phi,\psi )\mapsto  \frac{1}{(2 \pi)^m}\int_{F_\C}\langle (I\otimes \Fc_F) u, (\zeta, \lambda)\mapsto \ee^{i \langle\lambda, z-i \Phi(\zeta) \rangle }\phi(\zeta)\rangle \psi(z)\,\dd z\in \C 
	\]
	is separately continuous, so that it induces a distribution $\widetilde f$ on $E\times F_\C$ by Schwartz's kernel theorem (cf.,~\cite[Theorem 2]{Schwartz} or~\cite[Theorem 5.2.1]{Hormander}). In addition, it is clear that $\overline{\partial_{F_\C,w}}\widetilde f=0$ for every $w\in F_\C$, so that $\widetilde f_\phi$ coincides with a holomorphic function on $F_\C$ for every $\phi \in \Dc(E)$. 
	Furthermore,
	\[
	\begin{split}
		[\overline{\partial_{E,v}} \widetilde f]_\phi(z)&=- \frac{1}{(2 \pi)^m}\big\langle (I\otimes \Fc_F) u, (\zeta, \lambda)\mapsto \ee^{i \langle\lambda, z-i \Phi(\zeta) \rangle }\overline{\partial_{E,v}}\phi(\zeta)\big\rangle \\
			&=- \frac{1}{(2 \pi)^m}\big\langle (I\otimes \Fc_F) u,  (\zeta, \lambda)\mapsto (\overline{\partial_{E,v}}-\langle \lambda, \Phi(\zeta,v)\rangle) [\ee^{i \langle\lambda, z-i \Phi(\,\cdot\,) \rangle }\phi](\zeta)\big\rangle\\
			&=\frac{1}{(2 \pi)^m}\big\langle (I\otimes \Fc_F) (\overline{\partial_{E,v}} +i\Phi(\pr_E,v)\partial_F) u,  (\zeta, \lambda)\mapsto \ee^{i \langle\lambda, z-i \Phi(\zeta) \rangle }\phi(\zeta)\big\rangle\\
			&=\frac{1}{(2 \pi)^m}\big\langle (I\otimes \Fc_F) \overline{Z_v}u,  (\zeta, \lambda)\mapsto \ee^{i \langle\lambda, z-i \Phi(\zeta) \rangle }\phi(\zeta)\big\rangle=0
	\end{split}
	\]
	for every $v\in E$, for every $\phi \in \Dc(E)$, and for every $z\in F_\C$. Therefore, there is a unique $f\in \Hol(E\times F_\C)$ such that $\widetilde f=f\cdot \Hc^{2n+2m}$. In particular, if $\eta\in \Dc(E)$ is radial and has integral $1$, then
	\[
	f(\zeta,z)= \widetilde f_{\eta(\,\cdot\,-\zeta)}(z)= \frac{1}{(2 \pi)^m}\big\langle (I\otimes \Fc_F) u, (\zeta', \lambda)\mapsto \ee^{i \langle\lambda, z-i \Phi(\zeta') \rangle }\eta(\zeta'-\zeta)\big\rangle
	\]
	for every $(\zeta,z)\in E\times F_\C$. The same assertion holds also with $\eta_j\coloneqq2^{2 nj}\eta(2^j\,\cdot\,)$ for every $j\in \N$. In particular,
	\[
	\begin{split}
	\int_E \phi(\zeta) f(\zeta,z+i \Phi(\zeta))\,\dd \zeta&= \frac{1}{(2 \pi)^m}\int_E \phi(\zeta)\big\langle (I\otimes \Fc_F) u, (\zeta', \lambda)\mapsto \ee^{i \langle\lambda, z \rangle }\eta_j(\zeta'-\zeta)\big\rangle  \,\dd \zeta\\
		&= \frac{1}{(2 \pi)^m} \big\langle (I\otimes \Fc_F) u, (\phi*\eta_j)\otimes \ee^{i \langle\,\cdot\,, z \rangle }\big\rangle
	\end{split}
	\]
	 for every $\phi \in \Dc(E)$, for every $z\in F_\C$, and for every $j\in \N$. Passing to the limit for $j\to \infty$, we then infer that 
	 \[
	 \int_E \phi(\zeta) f(\zeta,z+i \Phi(\zeta))\,\dd \zeta=  \frac{1}{(2 \pi)^m} \big\langle (I\otimes \Fc_F) u, \phi\otimes \ee^{i \langle\,\cdot\,, z \rangle }\big\rangle,
	 \]
	 so that
	 \[
	 (I\otimes \Fc_F)u= [(\zeta,\lambda) \mapsto \Fc_F(f_0(\zeta,\,\cdot\,))(\lambda)]\cdot \Hc^{2n+m},
	 \]
	 whence $u=f_0\cdot \Hc^{2n+m}$.
	 
	 Next, observe that there is $N'\in \N$ such that
	 \[
	 \abs{\langle (I\otimes \Fc_F) u, \phi\rangle}\meg \sup_{(\zeta,x)\in \Nc} (1+\abs{\zeta}^2+\abs{x})^{N'} \sup_{j=0}^{N'} \abs{\phi^{(j)}(\zeta,x)}
	 \]
	 for every $\phi \in \Sc(\Nc)$. Then, it is readily verified that there is a constant $C'''>0$ such that
	 \[
	 \abs{f_0(\zeta,x)}\meg C''' (1+\abs{\zeta}^2+ \abs{x})^{N'}
	 \]
	 for every $(\zeta,x)\in \Nc$. Therefore, the functions
	 \[
	 (1+\abs{\zeta}^2)^{-N'} (1+\abs{\,\cdot\,})^{-(N'+m+1)} f_0(\zeta,\,\cdot\,),
	 \]
	 as $\zeta$ runs through $E$, are uniformly bounded in $L^1(F)$, so that the functions
	 \[
	 (1+\abs{\zeta}^2)^{-N'} \Fc_F(f_0(\zeta,\,\cdot\,)),
	 \]
	 as $\zeta$ runs through $E$, are uniformly bounded in the space $\Ec'^{N'+m+1}(\Nc)$ of distributions with compact support and order at most $N'+m+1$, and are supported in $K$.
	 Hence, the usual Paley--Wiener--Schwartz theorem (cf., e.g.,~\cite[Theorem 7.3.1]{Hormander} and its proof) shows that  there is a constant $C^{(4)}>0$ such that
	 \[
	 \begin{split}
	 	\abs{f(\zeta,z)}&\meg C^{(4)}(1+\abs{\zeta}^2)^{N'}(1+\abs{z-i\Phi(\zeta) })^{N'+m+1}\ee^{H_{K}(\rho(\zeta,z))}
	 \end{split}
	 \]
	 for every $(\zeta,z)\in E\times F_\C$. This leads to the conclusion. 
\end{proof}

\section{Paley--Wiener Theorems for Spaces of Schwartz Functions}\label{sec:4}

In this section, we prove a Paley--Wiener--Schwartz theorem for CR Schwartz functions.

\begin{deff}
	Let $K$ be a closed subset of $\overline{\Lambda_+}$. We define $	\Sc_{K}(\Nc)$ as the space of $ \phi\in \Sc(\Nc)$   such that $\pi_\lambda(\phi)=\chi_K(\lambda)\pi_\lambda(\phi) P_{\lambda,0}$ for every $\lambda\in F'$. 
\end{deff}

If $K$ is compact, then $\Sc_K(\Nc)=\Sc(\Nc)\cap \Oc_K(\Nc)$, thanks to Proposition~\ref{prop:5}.

\begin{teo}\label{teo:1}
	Let $K$ be a  compact convex subset of $F'$. Then, for every $f\in \Hol(E\times F_\C)$ such that for every  $N\in\N$ there is a constant $C_N>0$ such that
	\[
	\abs{f(\zeta,z)}\meg \frac{C_N}{(1+\abs{z})^N }\ee^{H_{K}(\rho(\zeta,z))} 
	\]
	for every $(\zeta,z)\in E\times F_\C$, one has $f_0\in  \Sc_{K\cap \overline{\Lambda_+}}(\Nc)$.
	
	Conversely, if $K\subseteq \overline{\Lambda_+}$, then for every $\phi\in  \Sc_{K}(\Nc)$  there is a unique $f \in \Hol(E\times F_\C)$ such that $f_0=\phi$.
	In addition, 
	\[
	f(\zeta,z)=\frac{2^{n-m}}{\pi^{n+m}}\int_{K} \tr(\pi_\lambda(\phi) \pi_\lambda(\zeta,0)^*) \ee^{\langle \lambda_\C, i z+\Phi(\zeta)\rangle} \abs{\Pfaff(\lambda)}\,\dd \lambda
	\]
	for every $(\zeta,z)\in E\times F_\C$, and for every $N\in \N$ there is $C_N>0$ such that
	\[
		\abs{f(\zeta,z)}\meg \frac{C_N}{(1+\abs{\zeta}^2+\abs{z})^N }\ee^{H_{K}(\rho(\zeta,z))} 
	\]
	for every $(\zeta,z)\in E\times F_\C$.
\end{teo}

\begin{proof}
	Take $f\in \Hol(E\times F_\C)$ such that for every $N\in\N$ there is a constant $C_N>0$ such that
	\[
	\abs{f(\zeta,z)}\meg \frac{C_N}{(1+\abs{z})^N }\ee^{H_{K}(\rho(\zeta,z))} 
	\]
	for every $(\zeta,z)\in F$. Let $X$ be a left-invariant differential operator   on $\Nc$, and denote by $\Xc$ its canonical extension to $E\times F_\C$. Observe that, by Cauchy estimates and the invariance of $\Xc$, for every $R>0$ there is a constant $C'_R>0$ such that
	\[
	\begin{split}
		\abs{(X f_0)(\zeta,x)}&=\abs{\Xc [f((\zeta,x+i \Phi(\zeta))\,\cdot\,  )](0,0)}\\
		&\meg C'_R \sup_{\abs{(\zeta',z')}\meg 	R} \abs{f((\zeta,x+i\Phi(\zeta))\cdot (\zeta',z'))}\\
		&\meg C'_R\sup_{\abs{(\zeta',z')}\meg R} \frac{C_N}{(1+\abs{z'+x+i\Phi(\zeta)+2 i\Phi(\zeta',\zeta)})^N} \ee^{H_{K}(\rho(\zeta',z'))}
	\end{split}
	\]
	for   every $(\zeta,x)\in \Nc$. If we take $R$ sufficiently small, then we may find a constant $C''_N>0$ such that
	\[
	\abs{(X f_0)(\zeta,x)}\meg \frac{C''_N}{(1+\abs{x}+\abs{\zeta}^2)^N }
	\]
	for every $(\zeta,x)\in \Nc$ and for every $N\in \N$.\footnote{Here we use the fact that the mapping $\zeta\mapsto \abs{\Phi(\zeta)}$ is proper and homogeneous of degree $2$.} The arbitrariness of $N$ and $X$ then implies that $f_0\in \Sc(\Nc)$. In addition, Proposition~\ref{prop:5} shows that $\Fc_F(f_0(\zeta,\,\cdot\,))$ is supported in $\overline{\Lambda_+}$ for almost every $\zeta\in E$, hence for \emph{every} $\zeta\in E$, by continuity.
	Now, Theorem~\ref{teo:2} implies that $f_0\in \Oc_{K}(\Nc)$, so that $f_0\in \Sc_{K\cap \overline{\Lambda_+}}(\Nc)$.
	
	Conversely, assume that $K\subseteq \overline{\Lambda_+}$, and take $\phi\in  \Sc_{K}(\Nc)$. Observe that $\phi=0$ is $K$ has an empty interior, so that we may assume that $\Lambda_+\neq \emptyset$, so that $d_\lambda=0$ for almost every $\lambda\in F'$. Then, by Theorem~\ref{teo:2}, there is a unique $f\in \Hol(E\times F_\C)$ such that $\phi=f_0$. In addition, 
	\[
	f(\zeta,z)=\frac{1}{(2\pi)^m} \int_K \Fc_F [\phi(\zeta,\,\cdot\,)](\lambda) \ee^{i \langle \lambda_\C,z\rangle}\,\dd \lambda
	\]
	for every $(\zeta,z)\in E\times F_\C$.  This implies that $\pi_\lambda(f_h)= \ee^{-\langle \lambda,h\rangle} \pi_\lambda(\phi)$ for every $h\in F$ and for every $\lambda\in F'$, so that Proposition~\ref{prop:2} implies that
	\[
	\begin{split}
	f (\zeta,z)&= \frac{2^{n-m}}{\pi^{n+m}} \int_K \tr(\pi_\lambda(\phi) \pi_\lambda(\zeta,\Re z)^*) \ee^{-\langle \lambda, \rho(\zeta,z)\rangle} \abs{\Pfaff(\lambda)}\,\dd \lambda\\
		&= \frac{2^{n-m}}{\pi^{n+m}} \int_K \tr(\pi_\lambda(\phi) \pi_\lambda(\zeta,0)^*) \ee^{\langle \lambda_\C, i z+\Phi(\zeta) \rangle} \abs{\Pfaff(\lambda)}\,\dd \lambda.
	\end{split}
	\]
	Now, observe that the mapping $\zeta\mapsto \phi(\zeta,\,\cdot\,)$ belongs to $\Sc(E; \Sc_{K}(F))$, so that the family $((1+\abs{\zeta}^2)^N f(\zeta,\,\cdot\,))_{\zeta\in E}$ is bounded in $\Sc_K(F)$ for every $N\in \N$.	In particular, the usual Paley--Wiener--Schwartz theorem, cf., e.g.,~\cite[Theorem 7.3.1]{Hormander} (and its proof) shows that for every $N\in \N$ there is a constant $C_N'''>0$ such that
	\[
	\abs{f(\zeta,z+i \Phi(\zeta))}\meg \frac{C'''_N}{(1+\abs{\zeta}^2)^{N} (1+\abs{z})^N} \ee^{H_{K}(\Im z)}
	\]
	for every $(\zeta,z)\in E\times F_\C$. To conclude, observe that there is $C^{(4)}\Meg 1$ such that $\abs{\Phi(\zeta)}\meg C^{(4)}\abs{\zeta}^2$ for every $\zeta\in E$, so that
	\[
	\begin{split}
		(1+\abs{\zeta}^2)^{N} (1+\abs{z-i \Phi(\zeta)})^N&\Meg (1+\abs{\zeta}^2+(2C^{(4)})^{-1}\abs{z-i \Phi(\zeta)})^N\\
			&\Meg (1+\abs{\zeta}^2/2+(2C^{(4)})^{-1}\abs{z})^N\\
			&\Meg (2 C^{(4)})^{-N} (1+\abs{\zeta}^2+\abs{z})^N
	\end{split}
	\]
	for every $(\zeta,z)\in E\times F_\C$.
\end{proof}

\section{Spaces of Schwartz Functions}\label{sec:5}

We now introduce a smaller auxialiry space of Schwartz functions, $\widetilde \Sc_K(\Nc)$, which can be more easily understood by means of the non-commutative Fourier transform, and then provide further insight into $\Sc_K(\Nc)$. In particular, we shall prove that $\Sc_K(\Nc)$ is the closure of the union of the $\Sc_H(\Nc)$, as $H$ runs through the set of compact subsets of the interior of $K$, a fact which is clear when $E=\Set{0}$, but more delicate in the general case.
In addition, we shall provide more insight into the structure of $\Oc_K(\Nc)$.

Recall that $W=\Set{\lambda\in F'\colon d_\lambda>d}$, with $d=\min_\lambda d_\lambda$.

\begin{deff}
	Let $K$ be a closed subset of $\overline{\Lambda_+}$.
	Then, we define
	\[
	\widetilde \Sc_K(\Nc)\coloneqq \Set{\phi \in \Sc(\Nc)\colon\forall \lambda\in F'\setminus W \:\: \pi_\lambda(\phi)=\chi_K(\lambda)P_{\lambda,0} \pi_\lambda(\phi) P_{\lambda,0}}.
	\]
	In addition, we define
	\[
	\Fc_\Nc\colon \widetilde \Sc_{\overline{\Lambda_+}}(\Nc)\ni \varphi\mapsto [\lambda \mapsto \tr(\pi_\lambda(\varphi) ) ] .
	\]
	Finally, we define $\Sc(F',K)$ as the space of elements of $\Sc(F')$ supported in $K$, endowed with the topology induced by $\Sc(F')$.
\end{deff}

\begin{prop}\label{prop:40}
	Let $K$ be a closed subset of $\overline{\Lambda_+}$. Then, the following hold:
	\begin{enumerate}
		\item[\textnormal{(1)}] $\Fc_\Nc$ induces an isomorphism of $\Sc_{K}(\Nc)$ onto $\Sc(F',K)$;
		
		\item[\textnormal{(2)}] for every $\psi\in \Sc(F',K)$ and for every $(\zeta,x)\in \Nc$,
		\[
		(\Fc_\Nc^{-1} \psi)(\zeta,x)=\frac{2^{n-m} }{\pi^{n+m}}\int_{K} \psi(\lambda)\abs{\Pfaff(\lambda)} \ee^{i\langle\lambda,x\rangle-\langle \lambda, \Phi(\zeta)\rangle}\,\dd \lambda;
		\]
		
		\item[\textnormal{(3)}] if $\varphi_1,\varphi_2\in \widetilde\Sc_{K}(\Nc)$, then $\varphi_1*\varphi_2\in \widetilde \Sc_{K}(\Nc)$ and 
		\[
		\Fc_\Nc(\varphi_1*\varphi_2)=(\Fc_\Nc \varphi_1)(\Fc_\Nc \varphi_2);
		\]
		
		\item[\textnormal{(4)}]  if $\phi_1,\phi_2\in \widetilde \Sc_{\overline{\Lambda_+}}(\Nc)$, then $\phi_1\phi_2\in \widetilde \Sc_{\overline{\Lambda_+}}(\Nc)$ and
		 \[
		 \Fc_\Nc(\phi_1 \phi_2)=\frac{2^{n-m}}{\pi^{n+m}}\frac{[\Fc_\Nc(\phi_1) \abs{\Pfaff(\,\cdot\,)}]*[\Fc_\Nc(\phi_2) \abs{\Pfaff(\,\cdot\,)}]}{\abs{\Pfaff(\,\cdot\,)}}.
		 \]
	\end{enumerate} 
\end{prop}

In particular,  convolution is commutative on $\widetilde \Sc_{K}(\Nc)$.
Observe that the mapping $\Fc_\Nc$ is essentially the (non-commutative) Fourier transform on $\Nc$, since
\[
\pi_\lambda(\varphi)= \Fc_\Nc(\varphi)(\lambda) P_{\lambda,0}
\]
for every $\lambda\in F'\setminus W$ and for every $\varphi \in \widetilde\Sc_{K}(\Nc)$ (and we may assume that $\Lambda_+\neq \emptyset$).

We first prove an auxiliary result, which we state in a more general form for later use.

\begin{deff}
	Let $K$ be a closed subset of $F'$. We define $\Gamma_K\coloneqq \Set{h\in F\colon H_{K}(h)<\infty}$.
	
	We denote by $\Oc_M(F')$ the set of $\phi \in C^\infty(F')$ which grow at most polynomially on $F'$ with every derivative.
\end{deff}

Observe that $\Gamma_K$ is the convex cone generated by the polar $K^\circ$ of $K$. 
In addition, $\Oc_M(F')$ is the space of (pointwise) multipliers of $\Sc(F')$.

Given a semi-norm $\nu$ on a vector space $X$, we shall also write $\norm{x}_\nu$ instead of $\nu(x)$, for every $x\in X$.

\begin{lem}\label{lem:6}
	Let $K$ be a closed convex subset of $\overline{\Lambda_+}$. For every $f\in \Sc_K(\Nc)$, for every $\tau\in \Oc_M(F')$, and for every $h\in \Gamma_K$, define $T(f,\tau,h)\in L^2(\Nc)$ so that
	\[
	\pi_\lambda(T(f,\tau,h))= \ee^{-\langle \lambda, h\rangle} \tau(\lambda) \pi_\lambda(f)
	\]
	for almost every $\lambda\in F'$.
	
	Then, for every $N_1$ there is $N_2\in \N$  such that for every $N_3\in \N$ there is  a continuous semi-norm $\nu$ on $\Sc(\Nc)$ such that
	\[
	\abs{T(f,\tau,h)(\zeta,x)}\meg  \frac{\norm{f}_\nu}{(1+\abs{x+i(h+\Phi(\zeta))})^{N_1}} \ee^{H_{K}(h)} \sup_{\lambda\in K} \sup_{j=1}^{N_2} \frac{\min(1,d(\lambda, \partial K))^{N_3}\abs{\tau^{(j)}(\lambda)}}{(1+\abs{\lambda})^{N_3}}
	\]
	for every $f\in \Sc_K(\Nc)$, for every $\tau\in \Oc_M(F')$, for every $h\in \Gamma_K$, and for  almost every $(\zeta,x)\in \Nc$.
\end{lem}

Before we pass to the proof, we need some elementary lemmas for which we could not find a reference.

\begin{lem}\label{lem:4}
	Let $K$ be a   closed convex subset of $F'$. Then, for every differential operator with constant coefficients on $F'$ and for every $N\in\N$, there is a continuous semi-norm $\nu$ on $\Sc(F')$ such that
	\[
	\abs{X \phi(\lambda)}\meg \frac{\min(1,d(\lambda, \partial K))^N}{(1+\abs{\lambda})^{N}} 		\norm{\phi}_\nu
	\]
	for every $\phi \in \Sc(F',K)$.
\end{lem}

\begin{proof}
	Observe that it will suffice to prove the stated inequality for every $\lambda$ in the interior $U$ of $K$, since $X\phi$ vanishes on the complement of $U$ in $F'$. In addition, it will suffice to consider the case in which $\partial K\neq \emptyset$, for otherwise either $K=\emptyset$ or $K=F'$, in which cases the assertion is clear. Further, it will suffice to prove the assertion for $\lambda\in U$ such that $d(\lambda, \partial K)\meg 1$.

	Then, take $\lambda\in U$ such that $d(\lambda, \partial K)\meg 1$, and take $\lambda_0\in \partial K$ such that $d(\lambda,\partial K)=\abs{\lambda-\lambda_0}$. Since $K$ is convex,  the mapping $t\mapsto (X\phi)(\lambda_0+ t (\lambda-\lambda_0))$ vanishes on $\R_-$, so that it vanishes with every derivative at $0$. Hence, by Taylor's formula, for every $N\in\N$,
	\[
	\begin{split}
		\abs{(X\phi)(\lambda)}&\meg \abs{\lambda-\lambda_0}^N \sup_{[\lambda_0,\lambda]} \abs{\partial_{\sgn(\lambda-\lambda_0)}^N X\phi}
		\meg d(\lambda, \partial K)^N \sup_{F'} \sup_{\abs{v'}=1} \abs{\partial_{v'}^N X \phi},
	\end{split}
	\]
	whence the result for $\abs{\lambda}\meg 2$. If, otherwise, $\abs{\lambda}\Meg 2$ (and $d(\lambda,\partial K)\meg 1$), then the preceding arguments show that, for every $N\in\N$,
	\[
	\begin{split}
		\abs{(X\phi)(\lambda)}\meg \frac{d(\lambda, \partial K)^N}{(1+\abs{\lambda})^{N}} \sup_{F'} (2+\abs{\,\cdot\,})^{N} \sup_{\abs{v'}=1} \abs{\partial_{v'}^N X \phi},
	\end{split}
	\]
	since $[\lambda_0,\lambda]\subseteq \overline B_{F'}(\lambda,1)$. The assertion follows.
\end{proof}

\begin{lem}\label{lem:5}
	Let $K$ be a closed convex \emph{cone} in $F'$.
	Then, there is a constant $C>0$ such that
	\[
	\langle \lambda, h\rangle \Meg C \abs{h} d(\lambda, \partial K)
	\]
	for every $\lambda\in K$ and for every $h\in K^\circ$.
\end{lem}

\begin{proof}
	We may assume that $K$ has a non-empty interior $U$ and that $K\neq F'$, for otherwise the assertion is trivial (with the usual convention $0\cdot\infty=0$).
	Then, by homogeneity, it will suffice to prove the assertion for $\lambda\in K\cap \partial B_{F'}(0,1)$ and $h\in K^\circ \cap \partial B_F(0,1)$. 
	
	Define $H\coloneqq \partial K\cap \partial B_{F'}(0,1)$, so that $H$ is a compact subset of $F'$. In addition, observe that, since $K\neq F' $, $K$ is contained in a half-space. Therefore, there is $\lambda_0\in U\cap \partial B_{F'}(0,1)$ such that $K\subseteq \Set{\lambda'\in F'\colon \langle \lambda' ,\lambda_0\rangle_{F'}\Meg 0}$.\footnote{Define $L\coloneqq \Set{\lambda'\in F'\colon \forall \lambda''\in K \:\:\langle \lambda',\lambda''\rangle_{F'}\Meg 0}$, and observe that $L$ is a closed convex cone not reduced to $0$, that $K= \Set{\lambda'\in F'\colon \forall \lambda''\in L\:\:\langle \lambda',\lambda''\rangle_{F'}\Meg 0}$, and that our assertion is equivalent to requiring that $U\cap L\neq \emptyset$. Assume by contradiction that $L\cap U=\emptyset$. Then,~	\cite[Theorem 11.3]{Rockafellar} implies that there is a non-zero $\lambda_1\in F'$ such that $\langle \lambda_1,\lambda'\rangle_{F'}\Meg0 $ for every $\lambda'\in K$ and $\langle \lambda_1,\lambda'\rangle_{F'}\meg0 $ for every $\lambda'\in L$. Then, $\lambda_1\in L\cap (-K)$, so that $\abs{\lambda_1}^2=\langle \lambda_1,\lambda_1\rangle\meg 0$: contradiction.} Let us prove that
	\[                                                                                                                                                                   
	U\cap \partial B_{F'}(0,1)=\Set{ \sgn((1-t)\lambda_0+t\lambda)\colon t\in [0,1), \lambda\in H},
	\]
	where $\sgn(\lambda)=\frac{\lambda}{\abs{\lambda}}$ for $\lambda\neq 0$, and $\sgn(0)=0$.
	The inclusion $\supseteq$ is clear. Conversely, take $\lambda\in U\cap \partial B_{F'}(0,1)$, $\lambda \neq \lambda_0$, and observe that the half-line $\lambda_0+\R_+(\lambda-\lambda_0)$ meets the hyperplane $ \Set{\lambda'\colon \langle \lambda' ,\lambda_0\rangle_{F'}=0}$, hence in particular $\partial K$.\footnote{Indeed, $\langle \lambda,\lambda_0\rangle_{F'}<1$, so that $\langle \lambda_0+t(\lambda-\lambda_0), \lambda_0\rangle_{F'}=1+t(\langle \lambda, \lambda_0\rangle_{F'}-1)\to -\infty$ for $t\to +\infty$.} Choose $t>1$ so that $\lambda'\coloneqq (1-t)\lambda_0+t\lambda\in \partial K$, and observe that $\abs{\lambda'}>1 $. Then, setting $t'\coloneqq \frac{\abs{\lambda'}}{\abs{\lambda'}+t-1}\in (0,1)$, we have
	\[
	\frac{t}{\abs{\lambda'}+t-1}\lambda=(1-t')\lambda_0+t' \sgn(\lambda')
	\]
	with $\frac{t}{t+\abs{\lambda'}-1}\in (0,1)$, so that  $\lambda=\sgn((1-t')\lambda_0+t' \sgn(\lambda'))$ with $\sgn(\lambda')\in H$.  Thus, also the inclusion $\subseteq$ is proved.

	Now, observe that there is a constant $C_1>0$ such that
	\[
	\langle \lambda_0, h\rangle \Meg C_1 
	\]
	for every $h\in K^\circ \cap \partial B_F(0,1)$, since $\lambda_0$ vanishes nowhere on the compact set $K^\circ \cap \partial B_F(0,1)$. 
	Then, for every $t\in [0,1)$ and for every $\lambda\in H$,
	\[
	\begin{split}
		\langle \sgn((1-t)\lambda_0+t\lambda),h \rangle&\Meg \frac{1-t}{\abs{(1-t)\lambda_0+t\lambda}}\langle \lambda_0,h\rangle \\
		&\Meg C_1\frac{1-t}{\abs{(1-t)\lambda_0+t\lambda}}\\
		&= C_1\frac{\abs{((1-t)\lambda_0+t\lambda)-\lambda}}{\abs{(1-t)\lambda_0+t\lambda}\abs{\lambda_0-\lambda}}\\
		&\Meg \frac{C_1}{2}\frac{d((1-t)\lambda_0+t\lambda,\partial K)}{\abs{(1-t)\lambda_0+t\lambda}}\\
		&=\frac{C_1}{2}d(\sgn((1-t)\lambda_0+t\lambda),\partial K)
	\end{split}
	\]
	for every $h\in K^\circ \cap B_F(0,1)$. The assertion follows.	
\end{proof}

\begin{proof}[Proof of Lemma~\ref{lem:6}.]
	We may assume that $\Lambda_+\neq \emptyset$, so that $d_\lambda=0$ for  every $\lambda\in F'\setminus W$.
	Take $f\in \Sc_K(\Nc)$ and $\tau\in \Oc_M(F')$. Observe first that
	\[
	\begin{split}
		\tr(\pi_\lambda(f)\pi_\lambda(\zeta,x)^*)=& \tr(\pi_\lambda(\zeta,x)^*\pi_\lambda(f) P_{\lambda,0})\\
		&=\langle \pi_\lambda(\zeta,x)^*\pi_\lambda(f) e_{\lambda,0}\vert e_{\lambda,0}\rangle\\
		&=\int_\Nc \langle \pi_\lambda((\zeta,x)^{-1} (\zeta',x')) e_{\lambda,0}\vert e_{\lambda,0}\rangle f(\zeta',x')\,\dd (\zeta',x')\\
		&=\int_E \ee^{-\langle \lambda_\C,\Phi(\zeta-\zeta')-i (x+2 \Im \Phi(\zeta,\zeta'))\rangle}\Fc_F(f(\zeta',\,\cdot\,))(\lambda)\,\dd \zeta'
	\end{split}
	\]
	for every $\lambda\in F'\setminus W$ and for every $(\zeta,x)\in \Nc$, thanks to Proposition~\ref{prop:1} and the properties of $\Sc_K(\Nc)$. Therefore, Proposition~\ref{prop:2} implies that
	\[
	\begin{split}
		T(f,\tau,h)(\zeta,x)
		&=\frac{2^{n-m}}{\pi^{n+m}}  \int_{K}  \ee^{-\langle \lambda,h\rangle}\abs{\Pfaff(\lambda)} \tau(\lambda)\int_E \ee^{-\langle \lambda_\C,\Phi(\zeta-\zeta')-i (x+2 \Im \Phi(\zeta,\zeta'))\rangle}\Fc_F(f(\zeta',\,\cdot\,))(\lambda)\,\dd \zeta'\,\dd \lambda
	\end{split}
	\]
	for every $(\zeta,x)\in \Nc$ and for every $h\in \Gamma_K$ (to see that the integral is well defined, observe that the mapping $\zeta'\mapsto \Fc_F(f(\zeta'\,\cdot\,))$ belongs to $\Sc(E;\Sc(F',K))$).

	We identify $F$ and $F'$ with $\R^m$ by means of dual orthonormal bases, to simplify the notation.
	Then, take $\alpha\in \N^m$, and observe that, for every $(\zeta,x)\in \Nc$ and for every $h\in \Gamma_K$,
	\[
	\begin{split}
		&(x+i(h+\Phi(\zeta)))^\alpha T(f,\tau,h)(\zeta,x)=(-i)^{\abs{\alpha}}\frac{2^{n-m}}{\pi^{n+m}}  \int_{K}  \Big(\partial^{\alpha}_\lambda \ee^{i\langle \lambda_\C, x+i(h+\Phi(\zeta))\rangle} \Big)\abs{\Pfaff(\lambda)} \tau(\lambda) \\
		&\qquad\qquad \times\int_E \ee^{-\langle \lambda_\C,\Phi(\zeta-\zeta')-\Phi(\zeta)-2 i  \Im \Phi(\zeta,\zeta')\rangle}\Fc_F(f(\zeta',\,\cdot\,))(\lambda)\,\dd \zeta'\,\dd \lambda\\
		&\qquad=i^{\abs{\alpha}}\frac{2^{n-m}}{\pi^{n+m}}  \int_{K}  \ee^{i\langle \lambda_\C, x+i(h+\Phi(\zeta))\rangle} \sum_{\beta_1+\beta_2+\beta_3+\beta_4=\alpha}  \frac{\alpha!}{\beta_1!\beta_2!\beta_3!\beta_4!}\partial^{\beta_1}_\lambda\abs{\Pfaff(\lambda)} (\partial^{\beta_2}\tau)(\lambda) \int_E\Phi(2\zeta-\zeta',\zeta')^{\beta_3}\\
		&\qquad\qquad \times \ee^{-\langle \lambda_\C,\Phi(\zeta-\zeta')-\Phi(\zeta)-2 i  \Im \Phi(\zeta,\zeta')\rangle}  \partial^{\beta_4}_{F'}\Fc_F(f(\zeta',\,\cdot\,))(\lambda)\,\dd \zeta'\,\dd \lambda,
	\end{split}
	\]
	where the integration by parts in the second equality can be applied since the mapping $\zeta'\mapsto\Fc_F(f(\zeta',\,\cdot\,))$ belongs to $\Sc(E; \Sc(F',K))$ and $\abs{\Pfaff(\,\cdot\,)}$ is polynomial on $\Lambda_+$.\footnote{Note that $\abs{\Pfaff(\lambda)}=\det J_\lambda$ for every $\lambda\in \Lambda_+$, with the notation of Section~\ref{sec:2:1}.}
	Then, fix $\beta_1,\beta_2,\beta_3,\beta_4\in \N^m$, and consider the function $I_{\beta_1,\beta_2,\beta_3,\beta_4}$ defined as follows:
	\[
	(\zeta,x)\mapsto  \int_{K}   \partial^{\beta_1}_\lambda \abs{\Pfaff(\lambda)}(\partial^{\beta_2}\tau)(\lambda) \int_E  \Phi(2\zeta-\zeta',\zeta')^{\beta_3} \ee^{-\langle \lambda_\C,h+\Phi(\zeta-\zeta')-i (x+2 \Im \Phi(\zeta,\zeta'))\rangle} \partial^{\beta_4}_{F'}\Fc_F(\phi(\zeta',\,\cdot\,))(\lambda)\,\dd \zeta'\,\dd \lambda.
	\]
	Then,
	\[
	(x+i(h+\Phi(\zeta)))^\alpha T(f,\tau,h)=i^{\abs{\alpha}}\frac{2^{n-m}}{\pi^{n+m}}   \sum_{\beta_1+\beta_2+\beta_3+\beta_4=\alpha} \frac{\alpha!}{\beta_1!\beta_2!\beta_3!\beta_4!}I_{\beta_1,\beta_2,\beta_3,\beta_4}
	\]
	for every $(\zeta,x)\in \Nc$, for every $h\in \Gamma_K$, and for every $\alpha\in \N^m$.
	Observe that Lemma~\ref{lem:4} implies that, for every $N\in\N$ and for every $\beta\in \N^m$, there are $C'_{N,\beta}, M_{N,\beta}\in \N$ such that
	\[
	\abs{\partial^{\beta}_{F'}\Fc_F(g(\zeta',\,\cdot\,))(\lambda)}\meg C'_{N,\beta}\frac{\min(1,d(\lambda,\partial K))^N}{(1+\abs{\lambda})^N} \sup_{x'\in F} (1+\abs{x'})^{M_{N,\beta}} \sup_{\abs{\gamma}\meg M_{N,\beta}} \abs{\partial_F^\gamma g(\zeta',x')}
	\]
	for every $\lambda\in K$, for every $\zeta'\in E$, and for every $g\in  \Sc_K(\Nc)$. Therefore,
	\[
	\begin{split}
		\abs{I_{\beta_1,\beta_2,\beta_3,\beta_4}(\zeta,x)}&\meg C'_{N,\beta_4}  \int_{K}  \frac{\partial^{\beta_1}_\lambda\abs{\Pfaff(\lambda)}\abs{(\partial^{\beta_2}\tau)(\lambda)}}{(1+\abs{\lambda})^N} \int_E  \abs*{\Phi(2\zeta-\zeta',\zeta')^{\beta_3}}\min(1,d(\lambda,\partial K))^N  \ee^{-\langle \lambda, h+\Phi(\zeta-\zeta')\rangle} \\
		&\qquad\times \sup_{x'\in F} (1+\abs{x'})^{M_{N,\beta_4}} \sup_{\abs{\gamma}\meg M_{N,\beta_4}} \abs{\partial_F^\gamma f(\zeta',x')} \,\dd \zeta'\,\dd \lambda
	\end{split}
	\]
	for every $\beta_1,\beta_2,\beta_3,\beta_4\in \N^m$, for every $(\zeta,x)\in \Nc$, for every $h\in \Gamma_K$,  and for every $N\in\N$.
	Now, observe that Lemma~\ref{lem:5} implies that there is a constant $c>0$ such that
	\[
	\ee^{-\langle \lambda_\C,\Phi(\zeta-\zeta')\rangle} \meg \ee^{-c d(\lambda,\partial \Lambda_+) \abs{\Phi(\zeta-\zeta')}}  
	\]
	for every $\lambda\in \Lambda_+$ and for every $\zeta,\zeta'\in E$. In addition, for every $\beta\in \N^m$ there is a constant $C''_\beta>0$ such that
	\[
	\abs{\Phi(2\zeta-\zeta',\zeta')^{\beta}}\meg C''_\beta (1+\abs{\Phi(\zeta-\zeta')})^{\abs{\beta}/2}(1+\abs{\zeta'})^{2 \abs{\beta}} 
	\]
	for every $\zeta,\zeta'\in E$.
	
	Therefore, for every $N\in\N$ there is a constant $C'''_N>0$ such that
	\[
	\begin{split}
		&\abs{\Phi(2\zeta-\zeta',\zeta')^{\beta}}\min(1,d(\lambda,\partial K))^N  \ee^{-\langle \lambda_\C,\Phi(\zeta-\zeta')\rangle}\\
		&\qquad\meg C''_\beta (1+\abs{\zeta'})^{2 \abs{\beta}} (1+\abs{\Phi(\zeta-\zeta')})^{\abs{\beta}/2}\min(1,d(\lambda,\partial \Lambda_+))^N  \ee^{-c d(\lambda,\partial \Lambda_+) \abs{\Phi(\zeta-\zeta')}}  \\
		&\qquad\meg C'''_N(1+\abs{\zeta'})^{2 \abs{\beta}}
	\end{split}
	\]
	for every $\beta\in \N^m$ with $\abs{\beta}\meg 2N$, for every $\lambda\in K$, and for every $\zeta,\zeta'\in E$.
	Therefore,
	\[
	\begin{split}
		\abs{I_{\beta_1,\beta_2,\beta_3,\beta_4}(\zeta,x)}&\meg C'_{2N,\beta_4} C'''_N \ee^{H_{K}(h)} \int_{K}  \frac{\partial^{\beta_1}_\lambda\abs{\Pfaff(\lambda)} \abs{(\partial^{\beta_2}\tau)(\lambda)} \min(1,d(\lambda, \partial K))^N}{(1+\abs{\lambda})^{2N}}\,\dd \lambda\\
		&\qquad\times\int_E (1+\abs{\zeta'})^{4 N} \sup_{x'\in F} (1+\abs{x'})^{M_{ 2N, \beta_4}} \sup_{\abs{\gamma}\meg M_{2N,\beta_4}} \abs{\partial_F^\gamma f(\zeta',x')} \,\dd \zeta'
	\end{split}
	\]
	for every $\beta_1,\beta_2,\beta_3,\beta_4\in \N^m$, for every $N\Meg \abs{\beta_3}/2$, for every $(\zeta,x)\in \Nc$, and for every $h\in \Gamma_K$. 
	Observe that the mapping
	\[
	f \mapsto \int_E (1+\abs{\zeta'})^{4 N} \sup_{x'\in F} (1+\abs{x'})^{M_{ N, \beta_3}} \sup_{\abs{\gamma}\meg M_{N,\beta_3}} \abs{\partial_F^\gamma f(\zeta',x')} \,\dd \zeta'
	\]
	is a continuous semi-norm on $\Sc(\Nc)$. The assertion then follows by the arbitrariness of $\alpha$.
\end{proof}

\begin{proof}[Proof of Proposition~\ref{prop:40}.]
	We may assume that $\Lambda_+\neq \emptyset$, so that $d_\lambda=0$ for every $\lambda\in F'\setminus W$.
	
	{ (1)--(2)} It will suffice to prove the assertion for $K=\overline{\Lambda_+}$.	
	Take $\phi \in\widetilde \Sc_{\overline{\Lambda_+}}(\Nc)$, and observe that 
	\[
	\phi(0,x)= \frac{2^{n-m}}{\pi^{n+m}} \int_{\Lambda_+} \Fc_\Nc(\phi)(\lambda) \ee^{i \langle \lambda, x\rangle} \abs{\Pfaff(\lambda)}\,\dd \lambda
	\]
	for every $x\in F$, thanks to Propositions~\ref{prop:1} and~\ref{prop:2}. Therefore,
	\[
	\Fc_\Nc(\phi) \abs{\Pfaff(\,\cdot\,)}=\frac{\pi^n}{2^n}  \Fc_F(\phi(0,\,\cdot\,))\in \Sc(F',\overline{\Lambda_+})
	\]
	since $\phi(0,\cdot\,)\in \Sc_{\overline{\Lambda_+}}(F)$, so that $\Fc_\Nc(\phi)\in \Sc(F',\overline{\Lambda_+})$ since $\abs{\Pfaff(\,\cdot\,)}$ is polynomial and vanishes nowhere on $\Lambda_+$ (cf.~Lemma~\ref{lem:4}). 
	
	Conversely, take $\psi \in \Sc(F',\overline{\Lambda_+})$, and define
	\[
	\phi(\zeta,x)\coloneqq \frac{2^{n-m} }{\pi^{n+m}}\int_{K} \psi(\lambda)\abs{\Pfaff(\lambda)} \ee^{i\langle\lambda,x\rangle-\langle \lambda, \Phi(\zeta)\rangle}\,\dd \lambda=\frac{2^n}{\pi^n}T(\Fc^{-1}_F (\psi \abs{\Pfaff(\,\cdot\,)}), 1, \Phi(\zeta))(x)
	\]
	for every $(\zeta,x)\in \Nc$,
	with the notation of Lemma~\ref{lem:6} (applied with $F$ in place of $\Nc$, that is, in the case $E=\Set{0}$). Observe that, if $X$ is a left-invariant differential operator on $\Nc$, then it is a differential operator with polynomial coefficients on $E\times F$. Hence, there are polynomials $P_1,\dots, P_k$ on $E$ and $Q_1,\dots, Q_k$ on $F'$ such that
	\[
	\begin{split}
	(X \phi)(\zeta,x)&= \frac{2^{n-m} }{\pi^{n+m}}\sum_{j=1}^k P_j(\zeta)\int_{K} \psi(\lambda)\abs{\Pfaff(\lambda)} Q_j(\lambda) \ee^{i\langle\lambda,x\rangle-\langle \lambda, \Phi(\zeta)\rangle}\,\dd \lambda\\
		&=\frac{2^n}{\pi^n}\sum_{j=1}^k P_j(\zeta) T(\Fc^{-1}_F (\psi \abs{\Pfaff(\,\cdot\,)}), Q_j, \Phi(\zeta))(x)
	\end{split}
	\]
	for every $(\zeta,x)\in \Nc$, so that the estimates of Lemma~\ref{lem:6} imply that $\phi \in \Sc(\Nc)$. Therefore $\phi \in \widetilde \Sc_{\overline{\Lambda_+}}(\Nc)$, since $\Fc_\Nc(\phi)= \psi$ by Propositions~\ref{prop:1} and~\ref{prop:2}.
	Since both $\widetilde \Sc_{\overline{\Lambda_+}}(\Nc)$ and $\Sc(F',\overline{\Lambda_+})$ are Fréchet spaces, by means of the closed graph theorem it is readily verified that $\Fc_\Nc$ is an isomorphism of Fréchet spaces.
	
	{(3)} Take $\varphi_1,\varphi_2\in \widetilde\Sc_K(\Nc)$, and observe that $\pi_\lambda(\varphi_1*\varphi_2)=\pi_\lambda(\varphi_1)\pi_\lambda(\varphi_2)$ for every $\lambda\in F'$. Since $\pi_\lambda(\varphi_j)= \Fc_\Nc(\varphi_j) P_{\lambda,0}$ for every $\lambda\in F'\setminus W$ and for every $j=1,2$, this implies that $\varphi_1*\varphi_2\in \widetilde\Sc_{K}(\Nc)$ and that
	\[
	\Fc_\Nc(\varphi_1*\varphi_2)=(\Fc_\Nc \varphi_1)(\Fc_\Nc \varphi_2),
	\]
	as we wished to show. 
	
	(4) This follows from (2) and from the facts that $\Sc(F',\overline{\Lambda_+})$ is a convolution algebra and that multiplication by $\abs{\Pfaff(\,\cdot\,)}$ induce an automorphism of $\Sc(F',\overline{\Lambda_+})$ (whose inverse is given by multiplication by $\abs{\Pfaff(\,\cdot\,)}^{-1}$), since $\abs{\Pfaff(\,\cdot\,)}$ is polynomial and vanishes nowhere on $\Lambda_+$ (cf.~Lemma~\ref{lem:4}).
\end{proof}

We now present some applications.
Recall that $\Pc=\Set{\lambda\in F'\colon \langle \lambda, \Phi\rangle\Meg 0}=\Phi(E)^\circ$.

\begin{prop}\label{prop:6}
	Let $K$ be a compact convex subset of $F'$. Then, $\Oc_K(\Nc)=\Oc_{K\cap \Pc}(\Nc)$.
\end{prop}

\begin{proof}
	\textsc{Step I.}  Assume first that $\Lambda_+\neq \emptyset$, so that $\Pc=\overline{\Lambda_+}$ by Proposition~\ref{prop:7}. Take $\phi \in \Oc_K(\Nc)$.
	Then, take $\eta\in C^\infty_c(\Lambda_+\cap B_{F'}(0,1))$, and define $\psi\coloneqq \Fc_\Nc(\eta)$ (cf.~Proposition~\ref{prop:6}).
	Define $\psi_j\coloneqq \psi(2^{-j}\,\cdot\,)$ for every $j\in \N$, so that $\psi_j\in \Oc_{\overline B_{F'}(0,2^{-j})}(\Nc)$ for every $j \in \N$ and $\psi_j\to \psi(0,0)$ boundedly and locally uniformly. Notice that we may assume that $\psi(0,0)=1$.
	
	Observe that $ \phi \psi_j\in L^2(\Nc)\cap \Oc_{K+ \overline B_{F'}(0, 2^{-j})}(\Nc)$ , so that $\Fc_F((\phi \psi_j)(\zeta,\,\cdot\,))$ is supported in $\Pc$ for every $\zeta\in E$ and for ever $j\in \N$, by Proposition~\ref{prop:3}.
	Hence, $\Fc_F(\psi(\zeta,\,\cdot\,))$ is supported in $\Pc$ for every $\zeta\in E$, so that $\phi\in \Oc_{K\cap \Pc}(\Nc)$.
	
	\textsc{Step II.} Now, consider the general case. Take $\phi \in \Oc_K(\Nc)$ and define $K'\coloneqq \overline{\bigcup_{\zeta\in E} \Supp{\Fc_F[\phi(\zeta,\,\cdot\,)]}}$. Observe that  it will suffice to show that $K'\subseteq \Set{\lambda\in F'\colon \langle \lambda , \Phi(v)\rangle \Meg 0}$ for every $v\in E$. 
	Then, take $v\in E$, and observe  $\Phi(v)=\Phi(v+v')$ for every $v'\in \Rc$, so that we may assume that $v\in \Rc^\perp$ and that $v\neq 0$, for otherwise the assertion is trivial. Then, define $\Nc_v\coloneqq \C v\times \R \Phi(v)$, so that $\Nc_v$ is a subgroup of $\Nc$ and $\phi\in \Oc_{K'_v}(\Nc)$, where $K'_v$ is the orthogonal projection of $K'$ into $(\R \Phi(v))^{\circ \perp}$, canonically identified with the dual of $\R \Phi(v)$ (cf.~Lemma~\ref{lem:1}). Then,~\textsc{step I} implies that $K'_v\subseteq \Set{\lambda\in (\R \Phi(v))^{\circ \perp}\colon \langle \lambda , \Phi(v)\rangle \Meg 0}$, so that $K'\subseteq  \Set{\lambda\in F'\colon \langle \lambda , \Phi(v)\rangle \Meg 0}$. The assertion follows.
\end{proof}

Now, let $K$ be a compact convex subset of $\Pc$.
We shall now describe further the structure of $\Oc_K(\Nc)$, reducing to the case in which $K\cap\Lambda_+$ has a non-empty interior (when $K\neq \emptyset$).
Define $F_{K,1}$ as the polar of the vector space generated by $K$. Notice that, when $K$ has a non-empty interior in $\Pc$, $F_{K,1}$ is the largest vector subspace of $F$ contained in the closed convex envelope of $\Phi(E)$, but in general may be larger. In addition, set
\[
E_{K,1}\coloneqq \Set{\zeta\in E\colon \forall \zeta'\in E\:\: \Phi(\zeta,\zeta')\in (F_{K,1})_\C},
\]
and observe that $\Nc_{K,1}\coloneqq E_{K,1}\times F_{K,1}$ is a normal subgroup of $\Nc$, as well as a quadratic CR manifold.  
Define $E_{K,2}\coloneqq E_{K,1}^\perp$, $F_{K,2}\coloneqq F_{K,2}^\perp$, and $\Phi_{K,2}\colon E_{K,2}\times E_{K,2}\to (F_{K,2})_\C$ so that $\Phi_{K,2}(\zeta,\zeta')-\Phi(\zeta,\zeta')\in (F_{K,1})_\C$ for every $\zeta,\zeta'\in E_{K,2}$. Then, $\Nc_{K,2}\coloneqq E_{K,2}\times F_{K,2}$ is a quadratic CR manifold and the canonical mapping $\Nc_{K,2}\to \Nc/\Nc_{K,1}$ (induced by the canonical projection) is an isomorphism of Lie groups. 
Observe that we may identify $F_{K,1}^\circ$ with the dual of $F_{K,2}$. With this identification, $\Pc\cap F_{K,1}^\circ$ becomes the polar of $\Phi_{K,2}(E_{K,2})$, since clearly
\[
\langle \lambda, \Phi(\zeta)\rangle=\langle \lambda, \Phi_{K,2}(\zeta')\rangle
\]
for every $\zeta\in E$, for every $\lambda\in F_{K,1}^\circ$, where $\zeta'\in E_{K,2}$ and $\zeta-\zeta'\in E_{K,1}$. In addition, $\Phi_{K,2}$ is non-degenerate. Therefore, $\Lambda_+(\Nc_{K,2})$ is the interior of $\Pc\cap F_{K,1}^\circ$ in $F_{K,1}^\circ$ by Proposition~\ref{prop:7}, so that the interior of $\Lambda_+(\Nc_{K,2})\cap K$ in $F_{K,1}^\circ$ is the interior of $K$ in $F_{K,1}^\circ$, and is therefore non-empty (unless $K=\emptyset$).

\begin{prop}\label{prop:10}
	Keep the preceding notation, and denote by $\Pol_{\CR}(\Nc_{K,1}; \Oc_K(\Nc_{K,2}))$ the space of CR polynomial mappings from $\Nc_{K,1}$ into $\Oc_K(\Nc_{K,2})$. Then, the mapping $\iota\colon \Pol_{\CR}(\Nc_{K,1};\Oc_K(\Nc_{K,2}))\to \Oc_K(\Nc)$ defined by
	\[
	\iota(P)(\zeta_1+\zeta_2,x_1+x_2)\coloneqq P_{\Phi(\zeta_1+\zeta_2)-\Phi(\zeta_1)-\Phi_{K,2}(\zeta_2)}(\zeta_1,x_1)(\zeta_2,x_2)
	\]
	for every $P\in \Pol_{\CR}(\Nc_{K,1};\Oc_K(\Nc_{K,2}))$ and for every $(\zeta_j,x_j)\in \Nc_{K,j}$, $j=1,2$, is an isomorphism.\footnote{If $Q$ is the holomorphic polynomial on $E_{K,1}\times (F_{K,1})_\C$ such that $Q_0=P$, we define $P_h\coloneqq Q_h$ for every $h\in F_{K,1}$. Since $\Phi(\zeta_1+\zeta_2)-\Phi(\zeta_1)-\Phi_{K,2}(\zeta_2)= 2 \Re \Phi(\zeta_1,\zeta_2)+(\Phi(\zeta_2)-\Phi_{K,2}(\zeta_2))\in F_{K,1}$ for every $\zeta_1\in E_{K,1}$ and for every $\zeta_2\in E_{K,2}$, $\iota(P)$ is well defined.}
\end{prop}

Notice that $\Oc_K(\Nc_{K,2})$ contains the non-trivial space $\widetilde \Sc_K(\Nc_{K,2})$, so that the preceding result cannot be improved.

\begin{proof}
	Take $\phi\in \Oc_K(\Nc)$, and denote by $f$ the unique element of $\Hol(E\times F_\C)$ such that $f_0=\phi$ (cf.~Theorem~\ref{teo:2}), so that there are $N\in \N$ and a constant $C>0$ such that
	\[
	\abs{f(\zeta,z)}\meg C(1+\abs{\zeta}^2+\abs{z})^N \ee^{H_K(\rho(\zeta,z))}
	\]
	for every $(\zeta,z)\in E\times F_\C$. Observe that, if $(\zeta_j,z_j)\in E_{K,j}\times (F_{K,j})_\C$, $j=1,2$, then
	\[
	\langle \lambda,\Im z_1+\Im z_2-\Phi(\zeta_1+\zeta_2)\rangle=\langle \lambda, \Im z_2-\Phi_{K,2}(\zeta_2)\rangle
	\]
	for every $\lambda\in F_{K,1}^\circ$, in particular for every $\lambda\in K$, so that	
	\[
	H_K(\rho(\zeta_1+\zeta_2,z_1+z_2))=H_K(\rho(\zeta_2,z_2)).
	\]
	Therefore,  the preceding estimates imply that the mapping
	\[
	E_{K,1}\times (F_{K,1})_\C\ni (\zeta_1,z_1)\mapsto f(\zeta_1+\zeta_2,z_1+z_2)\in \C
	\]
	is a holomorphic polinomial for every $(\zeta_2,z_2)\in E_{K,2}\times (F_{K,2})_\C$. In particular, the function
	\[
	P\colon \Nc_{K,1}\ni (\zeta_1,x_1)\mapsto [\Nc_{K,2} \ni (\zeta_2,x_2)\mapsto f(\zeta_1+\zeta_2,x_1+i\Phi(\zeta_1)+x_2+i\Phi_{K,2}(\zeta_2) )]
	\]
	belongs to $\Pol_{\CR}(\Nc_{K,2}; \Oc_K(\Nc_{K,2}))$, and $\iota(P)=\phi$.
	
	Conversely, take $P\in \Pol_{\CR}(\Nc_{K,2};\Oc_K(\Nc_{K,2}))$, and let us prove that $\iota(P)\in \Oc_K(\Nc)$. Let 
	\[
	\Es_{K,2}\colon\Oc_K(\Nc_{K,2})\to \Hol(E_{K,2}\times (F_{K,2})_\C)
	\]
	be the unique operator such that $(\Es_{K,2} \phi)_0=\phi$ for every $\phi \in \Oc_K(\Nc_{K,2})$. Let $Q$ be the unique holomorphic polynomial mapping $E\times (F_{K,1})_\C\to \Oc_K(\Nc_{K,2})$ such that $Q_0=P$. Define $f\in E\times F_\C\to \C$ so that
	\[
	f(\zeta_1+\zeta_2,z_1+z_2)\coloneqq \Es_{K,2}(Q(\zeta_1,z_1))(\zeta_2,z_2)
	\]
	for every $(\zeta_j,z_j)\in E_{K,j}\times (F_{K,j})_\C$, $j=1,2$, and observe that $f$ is holomorphic, since it is separately holomorphic in $(\zeta_1,z_1)$ and in $(\zeta_2,z_2)$. In addition, $f_0=\iota(P)$. Further, since $P$ is polynomial, there is $N_1\in \N$ so that the set of the $(1+\abs{\zeta_1}^2+\abs{x_2})^{-N_1} P(\zeta_1,x_2)$, as $(\zeta_1,x_1)$ runs through $\Nc_{K,1}$, is bounded in $\Oc_{K}(\Nc_{K,2})$.\footnote{For the topology induced by $\Sc'(\Nc_{K,2})$, for example. This topology is then naturally promoted to a much stronger one.} Therefore, Theorem~\ref{teo:2} implies that there are  $N_2\in \N$ and a constant $C>0$ such that
	\[
	\abs{f(\zeta_1+\zeta_2,z_1+z_2)}\meg C(1+\abs{\zeta_1}^2+\abs{z_1})^{N_1}  (1+\abs{\zeta_2}^2+\abs{z_2})^{N_2} \ee^{H_K(\rho(\zeta_2,z_2))}
	\] 
	for every $(\zeta_j,z_j)\in E_{K,j}\times (F_{K,j})_\C$, $j=1,2$. Since $H_K(\rho(\zeta_2,z_2))=H_K(\rho(\zeta_1+\zeta_2,z_1+z_2))$ by the above computations, this implies that $\iota(P)=f_0\in \Oc_K(\Nc)$, thanks to Theorem~\ref{teo:2}.
\end{proof}

\begin{cor}\label{cor:2}
	Let $K$ be a compact convex subset of $F'$ and take $p\in [1,\infty)$. Then, $\Oc_K(\Nc)\cap L^p(\Nc)\neq \Set{0}$ (resp.\ $\Oc_K(\Nc)\cap C_0(\Nc)\neq \Set{0}$) if and only if $K\cap \Lambda_+$ has a non-empty interior (equivalently, if and only if $\Phi$ is non-degenerate and $K\cap \Pc$ has a non-empty interior).
\end{cor}

Here, $C_0(\Nc)$ denotes the set of continuous functions on $\Nc$ which vanish at the point at infinity. 

\begin{proof}
	The assertion follows from Proposition~\ref{prop:10} if $K\cap\Lambda_+$ has an empty interior, since in this case $\Nc_{K,1}\neq \Set{(0,0)}$ and no non-zero polynomial on $\Nc_{K,1}$ belongs to $L^p$ or $C_0$.

	Conversely, assume that $K\cap \Lambda_+$ has a non-empty interior, so that, in particular, $\Lambda_+\neq \emptyset$. Then, $\Oc_K(\Nc)\cap L^p(\Nc)$ and $\Oc_K(\Nc)\cap C_0(\Nc)$ contain $\widetilde \Sc_K(\Nc)$, which is non-zero by Proposition~\ref{prop:40}.
\end{proof}

\begin{cor}
	Let $K$ be a compact subset of $\Pc$. Then, the mapping $\iota\colon \Oc_K(\Nc_{K,2})\cap L^\infty(\Nc_{K,2})\to \Oc_K(\Nc)\cap L^\infty(\Nc)$ defined by
	\[
	\iota(\phi)(\zeta_1+\zeta_2,x_1+x_2)\coloneqq \phi(\zeta_2,x_2)
	\]
	for every $(\zeta_j,x_j)\in \Nc_{K,j}$, $j=1,2$, is an isomorphism.
\end{cor}

\begin{proof}
	The assertion follows from Proposition~\ref{prop:10}, since the only bounded  CR polynomials on $\Nc_{K,2}$ are the constant functions.
\end{proof}

\begin{prop}\label{prop:41}
	Let $K$ be a closed  subset of $\overline{\Lambda_+}$. Then, the following hold:
	\begin{enumerate}
		\item[\textnormal{(1)}] $ \Sc_{K}(\Nc)$ is a left-invariant convolution algebra;
		
		\item[\textnormal{(2)}] $\Sc_{K_1}(\Nc) \Sc_{K_2}(\Nc)\subseteq  \Sc_{\overline{K_1+K_2}}(\Nc)$ for every two closed  subsets $K_1,K_2$ of $\overline{\Lambda_+}$.
	\end{enumerate}
\end{prop}

\begin{proof}
	{(1)} This is clear.
	
	{(2)} By Proposition~\ref{prop:5}, it will suffice to observe that, if $\phi_1\in \Sc_{K_1}(\Nc)$ and $\phi_2\in  \Sc_{K_2}(\Nc)$, then 
	\[
	\Fc_F((\phi_1\phi_2)(\zeta,\,\cdot\,))=\frac{1}{(2\pi)^m} \Fc_F(\phi_1(\zeta,\,\cdot\,))*\Fc_F( \phi_2(\zeta,\,\cdot\,))
	\]
	is supported in $\overline{K_1+K_2}$ for every $\zeta\in E$.
\end{proof}

\begin{deff}
	Let $K$ be a closed convex subset of $\overline{\Lambda_+}$ with interior $U$. Then, we shall define $\Sc_{K,c}(\Nc)\coloneqq \Sc(\Nc)*\Fc_\Nc^{-1}(C^\infty_c(U))$.
\end{deff}

Notice that $\Sc_{K,c}(\Nc)$ is the union of the $\Sc_H(\Nc)$ as $H$ runs through the set of compact subsets of $U$.

\begin{teo}\label{prop:9}
	Let $K$ be a closed convex subset of $\overline{\Lambda_+}$. 
	Then, $\Sc_K(\Nc)$ is the closure of $\Sc_{K,c}(\Nc)$  in $\Sc(\Nc)$. 
\end{teo}

With the notation of~\cite[Section 4.1]{CalziPeloso}, this proves that  $\widetilde \Sc_\Omega(\Nc)$ is the closure of $\Sc_{\Omega,L}(\Nc)$ in $\Sc(\Nc)$.

\begin{proof}
	We may assume that $K$ has a non-empty interior, so that, in particular, $\Lambda_+\neq \emptyset$ and $d_\lambda=0$ for every $\lambda\in F'\setminus W$.
	
	\textsc{Step I.} Assume that the assertion has been proved when $K$ is compact, and let us prove the statement in the general case. Take $f\in \Sc_K(\Nc)$. Observe that there exists $\phi \in \Sc(\Nc)$ such that $P_{\lambda,0}\pi_\lambda(\phi)=\eta(\lambda)P_{\lambda,0}$ for every $\lambda\in F'$, where $\eta\in C^\infty_c(F')$ and $\eta=1$ on  $ B_{F'}(0,1)$. Indeed, it suffices to take $\phi=\Kc(\theta)$, with the notation of Section~\ref{sec:1:3}, for some $\theta \in C^\infty_c(\R)$ such that $\theta=1$ on a suitable neighbourhood of $0$.
	Then $\eta$ is supported in $B_{F'}(0,R)$ for some $R>0$, so that 
	\[
	\pi_\lambda(f*\phi_j)=\eta(2^{-j}\,\cdot\,) \pi_\lambda(f)=\chi_{K\cap B_{F'}(0, 2^j R)} \pi_\lambda(f*\phi_j) P_{\lambda,0}
	\]
	for every $\lambda\in F'\setminus W$ and for every $j\in \N$, where $\phi_j\coloneqq 2^{j(n+m)} \phi(2^{j}\,\cdot\,)$. Then, the assumption shows that $f*\phi_j$ belongs to the closure of $\Sc_{K\cap B_{F'}(0,2^j R),c}(\Nc)$ in $\Sc(\Nc)$, hence to the closure of $\Sc_{K,c}(\Nc)$ in $\Sc(\Nc)$.
	Since clearly $f*\phi_j$ converges to $f$ in $\Sc(\Nc)$, the assertion follows.

	\textsc{Step II.} Let us prove the statement when $K$ is compact.	
	Define $K_\eps\coloneqq \Set{\lambda\in K\colon B_{F'}(\lambda,\eps)\subseteq K}$ for every $\eps>0$, so that $(K_\eps)_{\eps>0}$ is a decreasing family of compact convex subsets of $K$ whose union is the interior of $K$. 
	Fix $\psi\in C^\infty_c(B_{F'}(0,1))$, and define $\psi_\eps\coloneqq \eps^{-m} \psi(\eps^{-1}\,\cdot\,)$ and $\tau_\eps\coloneqq \chi_{K_{\eps/2}}*\psi_{\eps/4}$ for every $\eps>0$, so that 
	\[
	\chi_{K_\eps}\meg \tau_\eps\meg \chi_{K_{\eps/4}}
	\]
	and
	\[
	\norm{\partial^\alpha \tau_\eps}_\infty\meg (\eps/4)^{-\abs{\alpha}}\norm{\tau}_\infty\norm{\partial^\alpha\psi}_1
	\]
	for every $\alpha\in \N^m$ and for every $\eps>0$ (identifying $F'$ with $\R^m$ by means of an orthonormal basis).
	
	Fix $f\in \Sc_K(\Nc)$, and define
	\[
	f_\eps\coloneqq f*\Fc_\Nc^{-1}(\tau_\eps)
	\]
	for every $\eps>0$, so that $f_\eps\in \Sc_{K,c}(\Nc)$ for every $\eps>0$. In addition, Proposition~\ref{prop:2}  implies that $(f_\eps)$ converges to $f$ in $L^2(\Nc)$, for $\eps\to 0^+$, so that it will suffice to prove that the family $(f_\eps)_{\eps\in (0,1]}$ is bounded in $\Sc(\Nc)$.
	Observe that the preceding estimates and Lemma~\ref{lem:6} (choosing $N_3\Meg N_2$) imply that the family $(P f_\eps)_{\eps\in (0,1]}$ is bounded in $L^\infty(\Nc)$ for every polynomial $P$ on $\Nc$. 
	
	In order to complete the proof, take $\phi_j$, $j\in \N$, as in \textsc{Step I}, and observe that $f_\eps=f_\eps*\phi_j$ for every $\eps>0$, provided that $K\subseteq B_{F'}(0, 2^j)$, hence for $j$ sufficiently large. It is then clear that the $f_\eps$ are uniformly bounded in $\Sc(\Nc)$, as $\eps$ runs through $(0,1]$.	The proof is complete.
\end{proof}

\end{document}